\documentclass{amsart}
\usepackage{amsfonts,amsmath,amsthm,amssymb,color}
\usepackage[all]{xy}
\usepackage[colorlinks=true, linkcolor=blue, citecolor=blue, pagebackref]{hyperref} 
\usepackage{faktor}
\usepackage{graphicx}

\DeclareMathOperator*{\colim}{colim}
\numberwithin{equation}{section}

\begin{document}

\newtheorem{theorem}{Theorem}[section]
\newtheorem{axiom}[theorem]{Axiom}
\newtheorem{lemma}[theorem]{Lemma}
\newtheorem{proposition}[theorem]{Proposition}
\newtheorem{corollary}[theorem]{Corollary}

\theoremstyle{definition}
\newtheorem{definition}[theorem]{Definition}
\newtheorem{example}[theorem]{Example}

\theoremstyle{remark}
\newtheorem{remark}[theorem]{Remark}

\newenvironment{magarray}[1]
{\renewcommand\arraystretch{#1}}
{\renewcommand\arraystretch{1}}

\newenvironment{acknowledgement}{\par\addvspace{17pt}\small\rm
\trivlist\item[\hskip\labelsep{\it Acknowledgement.}]}
{\endtrivlist\addvspace{6pt}}

\newcommand{\quot}[2]{
{\lower-.2ex \hbox{$#1$}}{\kern -0.2ex /}
{\kern -0.5ex \lower.6ex\hbox{$#2$}}}

\newcommand{\mapor}[1]{\smash{\mathop{\longrightarrow}\limits^{#1}}}
\newcommand{\mapin}[1]{\smash{\mathop{\hookrightarrow}\limits^{#1}}}
\newcommand{\mapver}[1]{\Big\downarrow
\rlap{$\vcenter{\hbox{$\scriptstyle#1$}}$}}
\newcommand{\liminv}{\smash{\mathop{\lim}\limits_{\leftarrow}\,}}

\newcommand{\sSet}{\mathbf{sSet}}
\newcommand{\Set}{\mathbf{Set}}
\newcommand{\Art}{\mathbf{Art}}
\newcommand{\CDGA}{\mathbf{CDGA}}
\newcommand{\CGA}{\mathbf{CGA}}
\newcommand{\DGCA}{\mathbf{CDGA}}
\newcommand{\sCA}{\mathbf{sCA}}
\newcommand{\Top}{\mathbf{Top}}
\newcommand{\DGMod}{\mathbf{DGMod}}
\newcommand{\DGLA}{\mathbf{DGLA}}
\newcommand{\DGArt}{\mathbf{DGArt}}

\newcommand{\solose}{\Rightarrow}
\newcommand{\PSI}{\Psi\mathbf{Sch}_I(\mathbf{M})}
\newcommand{\PSJ}{\Psi\mathbf{Sch}_J(\mathbf{M})}
\newcommand{\Sym}{\mbox{Sym}}

\newcommand{\specif}[2]{\left\{#1\,\left|\, #2\right. \,\right\}}

\renewcommand{\bar}{\overline}
\newcommand{\de}{\partial}
\newcommand{\debar}{{\overline{\partial}}}
\newcommand{\per}{\!\cdot\!}
\newcommand{\Oh}{\mathcal{O}}
\newcommand{\sA}{\mathcal{A}}
\newcommand{\sB}{\mathcal{B}}
\newcommand{\sC}{\mathcal{C}}
\newcommand{\sF}{\mathcal{F}}
\newcommand{\sG}{\mathcal{G}}
\newcommand{\sH}{\mathcal{H}}
\newcommand{\sI}{\mathcal{I}}
\newcommand{\sL}{\mathcal{L}}
\newcommand{\sM}{\mathcal{M}}
\newcommand{\sN}{\mathcal{N}}
\newcommand{\sP}{\mathcal{P}}
\newcommand{\sU}{\mathcal{U}}
\newcommand{\sV}{\mathcal{V}}
\newcommand{\sX}{\mathcal{X}}
\newcommand{\sY}{\mathcal{Y}}
\newcommand{\sW}{\mathcal{W}}
\newcommand{\Ni}{\mathcal{N}_I}
\newcommand{\Nj}{\mathcal{N}_J}
\newcommand{\bM}{\mathbf{M}}
\newcommand{\bC}{\mathbf{C}}
\newcommand{\PM}{\Psi\mathbf{Mod}}
\newcommand{\QCoh}{\mathbf{QCoh}}
\newcommand{\DGSch}{\mathbf{DGSch}}
\newcommand{\DGAff}{\mathbf{DGAff}}

\newcommand{\Aut}{\operatorname{Aut}}
\newcommand{\Cof}{\operatorname{Cof}}
\newcommand{\Imm}{\operatorname{Imm}}
\newcommand{\Mor}{\operatorname{Mor}}
\newcommand{\Def}{\operatorname{Def}}
\newcommand{\D}{\operatorname{D}}
\newcommand{\Ho}{\operatorname{Ho}}
\newcommand{\Mod}{\operatorname{Mod}}
\newcommand{\Hom}{\operatorname{Hom}}
\newcommand{\Hilb}{\operatorname{Hilb}}
\newcommand{\HOM}{\operatorname{\mathcal H}\!\!om}
\newcommand{\DER}{\operatorname{\mathcal D}\!er}
\newcommand{\Spec}{\operatorname{Spec}}
\newcommand{\Der}{\operatorname{Der}}
\newcommand{\Tor}{{\operatorname{Tor}}}
\newcommand{\Ext}{{\operatorname{Ext}}}
\newcommand{\End}{{\operatorname{End}}}
\newcommand{\END}{\operatorname{\mathcal E}\!\!nd}
\newcommand{\Image}{\operatorname{Im}}
\newcommand{\coker}{\operatorname{coker}}
\newcommand{\tot}{\operatorname{tot}}
\newcommand{\Id}{\operatorname{Id}}
\newcommand{\cone}{\operatorname{cone}}
\newcommand{\cocone}{\operatorname{cocone}}
\newcommand{\cyl}{\operatorname{cyl}}
\newcommand{\id}{\operatorname{id}}
\newcommand{\mA}{\mathfrak{m}_{A}}

\renewcommand{\Hat}[1]{\widehat{#1}}
\newcommand{\dual}{^{\vee}}
\newcommand{\desude}[2]{\dfrac{\de #1}{\de #2}}

\newcommand{\A}{\mathbb{A}}
\newcommand{\SExt}{\mathbb{S}\mbox{Ext}}
\newcommand{\N}{\mathbb{N}}
\newcommand{\R}{\mathbb{R}}
\newcommand{\Z}{\mathbb{Z}}
\renewcommand{\H}{\mathbb{H}}
\renewcommand{\L}{\mathbb{L}}
\newcommand{\proj}{\mathbb{P}}
\newcommand{\K}{\mathbb{K}\,}
\newcommand\C{\mathbb{C}}
\newcommand\Del{\operatorname{Del}}
\newcommand\Tot{\operatorname{Tot}}
\newcommand\Grpd{\mbox{\bf Grpd}}

\newcommand\é{\'e}
\newcommand\è{\`e}
\newcommand\à{\`a}
\newcommand\ì{\`i}
\newcommand\ù{\`u}
\newcommand\ò{\`o }


\newcommand{\rh}{\rightarrow}
\newcommand{\contr}{{\mspace{1mu}\lrcorner\mspace{1.5mu}}}

\newcommand{\bi}{\boldsymbol{i}}
\newcommand{\bl}{\boldsymbol{l}}

\newcommand{\MC}{\operatorname{MC}}
\newcommand{\TW}{\operatorname{TW}}

\title{Formal deformation theory in left-proper model categories}
\date{\today}
\author{Marco Manetti}
\address{\newline
Universit\`a degli studi di Roma La Sapienza,\hfill\newline
Dipartimento di Matematica \lq\lq Guido
Castelnuovo\rq\rq,\hfill\newline
P.le Aldo Moro 5,
I-00185 Roma, Italy.}
\email{manetti@mat.uniroma1.it}
\urladdr{www.mat.uniroma1.it/people/manetti/}

\author{Francesco Meazzini}
\email{meazzini@mat.uniroma1.it}
\urladdr{www.mat.uniroma1.it/people/meazzini/}

\begin{abstract}
We develop the notion of deformation of a morphism in a left-proper model category. As an application 
we provide a geometric/homotopic description of deformations of commutative (non-positively) graded differential algebras over a local DG-Artin ring. 
\end{abstract}


\maketitle

\tableofcontents

\section*{Introduction}

This is the first of a series of papers devoted to the use of model category theory in the study of deformations of algebraic schemes and morphisms between them. 
In doing this we always try to reduce the homotopic and simplicial background at minimum, with the aim to be concrete and accessible to a wide community,  especially to everyone having a classical background in algebraic geometry and deformation theory.   

In order to explain the underlying ideas it is useful to sketch briefly their 
evolution, from the very beginning to the present form: needless to say that our plan of work is still fluid and several changes are possible in the near future. 

A very useful principle in deformation theory is that over a field of characteristic $0$ every deformation problem is controlled by a differential graded Lie algebra, according to the general and well understood construction of Maurer-Cartan modulus gauge action, see e.g. \cite{GM,Man}. 
As properly stated in \cite{Manin17},
the explicit construction of the relevant DG-Lie algebra controlling a given problem requires creative
thinking and the study of instructive examples existing in
the literature. 

For an affine scheme, it is well known and easy to prove that the DG-Lie algebra of derivations of a multiplicative Tate-Quillen resolution controls its deformations, since 
the Maurer-Cartan elements correspond to perturbation of the differential of the resolution. According to Hinich \cite{H04} the same recipe extends to (non-positively graded) DG-affine schemes and give a good notion of deformations of such objects over a (non-positively graded) differential graded local Artin ring (see also \cite[Section 4]{Man2} for a partial result in this direction). 

This  example is very instructive and suggests that for general separated schemes, the right DG-Lie algebra controlling deformations should be constructed by taking derivations of a Palamodov resolvent (possibly of special kind). 
Here the problem to face is that a Palamodov resolvent, as classically defined \cite{BF,Pal}, carries inside a quite complex  combinatorial structure, that leads to  very complicated computations in every attempt to 
prove the desired results. 

The key idea to overcome this difficulty is to interpret this combinatorics 
as the property of being cofibrant in a suitable model category, and then use the various lifting and factorization axioms of model categories in order to provide clear  and conceptually easier proofs.
However, it is our opinion that this approach works very well and gain new additional insight  whenever 
also the deformation theory of affine schemes is revisited in the framework of model categories. 
Since every multiplicative Tate-Quillen resolution of a commutative algebra is a special kind of cofibrant replacement in the category   $\DGCA_{\K}^{\le 0}$ of differential graded commutative algebras in non-positive degrees, equipped with the projective model structure, it is convenient to 
express, as much as possible, the notion of deformation in terms of the model structure. 
This will be quite easy for the condition of flatness (Definition~\ref{def.flatmorphism}, modelled on the notion of DG-flatness of \cite{A-F}) and for the 
local Artin ring version of Nakayama's lemma (Definition~\ref{def.M(K)}).

The first main result of this paper is to define a ``good'' formal deformation theory of a morphism on every model category in which every cofibration is flat: several left-proper model categories used in concrete applications have this property, included $\DGCA_{\K}^{\le 0}$. 
A remarkable fact is that the deformation theory of a morphism is homotopy invariant: more precisely given morphisms 
$K\xrightarrow{f}X \xrightarrow{g}Y$ with $g$ a weak equivalence, then the two morphisms $f$ and $gf$ have the same deformation theory; this allows in particular the possibility to restrict 
our attention to deformations of cofibrations.

The second main result is the proof that in the category  $\DGCA_{\K}^{\le 0}$ our general notion 
of deformation is equivalent to the notion introduced by Hinich and in particular 
gives the classical notion of deformation when restricted to algebras concentrated in degree 0. 
The main ingredient of the proof, that we consider of independent interest, is that 
both the left and right lifting properties and the (C-FW), (CW-F) factorization properties are unobstructed in the sense of \cite{Man2}, i.e., can be lifted along every surjective morphism of DG-local Artin rings (Theorems~\ref{thm.slashboxes}, 
\ref{thm.liftingfactorization} and \ref{thm.liftingfactorization2}).
The case of lifting properties is easy, while the unobstructedness  of (C-FW) and (CW-F) factorizations are quite involved and are proved as a consequence of a non-trivial technical result about liftings of trivial idempotents in cofibrant objects (Theorem~\ref{thm.liftingidempotents}): all this technical results will be extremely important in our next paper in order to treat in an easy way deformations of general separated algebraic schemes.

\section{Notation and preliminary results} 

The general theory is carried out on 
a fixed model category $\bM$, although the main relevant examples for the  applications of this paper are the  categories $\DGCA_{\K}$ of differential graded commutative algebras  over a field $\K$ of characteristic 0, and its full subcategory $\DGCA_{\K}^{\le 0}$ of algebras concentrated in non-positive degrees,
equipped with the projective model structure (\cite{BG76}, \cite[V.3]{GelfandManin}): in both categories 
weak equivalences are the quasi-isomorphisms and  cofibrations are the retracts of semifree extensions. Fibrations in  $\DGCA_{\K}$  (resp.: $\DGCA_{\K}^{\le 0}$) are the surjections (resp.: surjections in strictly negative degrees). In particular a morphism in $\DGCA_{\K}^{\le 0}$ is a weak equivalence (resp.: cofibration, trivial fibration) if and only if has the same property as a morphism in $\DGCA_{\K}$.

Starting from Section~\ref{section.deformationmorphism} we shall assume that $\bM$ satisfies the additional property introduced in Definition~\ref{def.SLP}.

For every object $A\in \bM$ we shall denote by  $A\!\downarrow\!\bM$ (or equivalently by $\bM_A$) the model undercategory of maps $A\to X$ in $\bM$, and by $\bM\!\downarrow\! A$ the overcategory of maps $X\to A$, \cite[p. 126]{Hir03}. Notice that for every $f\colon A\to B$ we have $(A\downarrow \bM)\downarrow f=f\downarrow(\bM\downarrow B)$.

Every morphism $f\colon A\to B$ in $\mathbf{M}$ induces two functors:
\[ f^{\ast}=-\circ f\colon  \mathbf{M}_B\to \mathbf{M}_A,\qquad (B\to X)\mapsto (A\xrightarrow{f}B\to X),\]
\[ f_{\ast}=-\amalg_A B\colon \mathbf{M}_A\to \mathbf{M}_B,\qquad X\mapsto X\amalg_AB\,.\]
According to the definition of the model structure in the undercategories of $\bM$, a morphism $h$ in $\bM_B$ is a weak equivalence (respectively: fibration, cofibration) if and only if $f^{\ast}(h)$ is a weak equivalence (respectively: fibration, cofibration), see~\cite[p. 126]{Hir03}.

For notational simplicity, in the diagrams we adopt the following labels about maps: $\sC$=cofibration, $\sF$=fibration, $\sW$=weak equivalence, $\sC\sW$=trivial cofibration, $\sF\sW$=trivial fibration. We also adopt the labels $\lrcorner$ for denoting pullback (Cartesian) squares, and $\ulcorner$ for pushout (coCartesian) squares.

\begin{definition}\label{def.idempotent}
An idempotent in $\mathbf{M}$ is an endomorphism $e\colon Z\to Z$ such that $e\circ e = e$. 
We shall say that $e$ is a trivial idempotent if it is also a weak equivalence.
The fixed locus $\iota\colon F_e\to Z$ of an idempotent $e\colon Z\to Z$ is the 
limit of the diagram
\[ \xymatrix{	Z \ar@/^/[r]^{\id_Z} \ar@/_/[r]_{e} & Z	} \] 
or, equivalently, the fibred product of the cospan 
\[ Z\xrightarrow{(\id_Z,\id_Z)}Z\times Z\xleftarrow{(e,\id_Z)}Z\,.\]
\end{definition}

\begin{lemma}\label{prop.fixedloci}
Let $e\colon Z\to Z$ be an idempotent in a model category $\mathbf{M}$ with fixed locus $\iota\colon F\to Z$. Then the following holds.
\begin{enumerate}
\item There exists a retraction
\[ F \xrightarrow{\iota} Z\xrightarrow{p} F \]
such that $\iota p = e$. In particular $p$ is a retract of $e$ and $pe=p$, $e\iota=\iota$.
If $e$ is a trivial idempotent, then  $\iota$ and $p$ are weak equivalences.

\item If there exists a retraction
\[ X \xrightarrow{i} Z\xrightarrow{q} X \]
such that $iq= e$  then $X\xrightarrow{i} Z$ is canonically isomorphic to the fixed locus of $e$.

\item The fixed locus of idempotents commutes with pushouts, i.e., for every span $Z\xleftarrow{f} A\to B$ 
and for every idempotent $e\colon Z\to Z$ such that $ef=f$,  
the fixed locus of the induced idempotent $e'\colon Z\amalg_AB\to Z\amalg_AB$ is naturally isomorphic to 
$\iota'\colon F\amalg_AB\to Z\amalg_AB$.
\end{enumerate}
\end{lemma}

\begin{proof} The first item is an immediate consequence of the universal property of limits applied to the diagram 
\[ \begin{matrix}\xymatrix{&Z\ar[dl]_e\ar[dr]^e&\\	
Z \ar@/^/[rr]^{\id_Z} \ar@/_/[rr]_{e} && Z	}\end{matrix}\;. \] 
For the second item, since $qe=qiq=q$, $ei=iqi=i$, we have that the two morphisms 
\[ X\xrightarrow{i}Z\xrightarrow{p}F,\qquad F\xrightarrow{\iota}Z\xrightarrow{q}X,\]
are one the inverse of the other.

In the the last item, the morphism $f$ lifts to a morphism $A\to F$ and the proof  
follows immediately from the fact  that retractions are preserved by pushouts.
\end{proof}

\bigskip
\section{Flatness in model categories}

Let $\bM$ be a model category  and let $\sG$ be a class of morphisms of $\bM$ containing 
all the isomorphisms and such that $\sG$ is closed under composition.

\begin{definition}\label{def.Gcofibration} 
A morphism $f\colon A\to B$ in $\bM$ is called
a $\sG$-\textbf{cofibration} if for every  $A\to M\xrightarrow{g}N$ with $g\in \sG$, the pushout morphism  
\[ M\amalg_AB\xrightarrow{\quad}N\amalg_AB \]
belongs to $\sG$.
\end{definition}

\begin{example}
When $\sG$ is exactly the class of isomorphisms, then every morphism is a $\sG$-cofibration.
\end{example}

\begin{remark}
Since finite colimits are defined by a universal property, they are defined up to isomorphism: therefore the assumption on the class $\sG$ are required in order to have that the notion of $\sG$-cofibration makes sense.   
\end{remark}

\begin{lemma}\label{lem.stability} 
In the situation of Definition~\ref{def.Gcofibration}, the class of $\sG$-cofibrations  contains the isomorphisms and is closed under composition and pushouts. If $\sG$ is closed under retractions, then the same holds for $\sG$-cofibrations.
\end{lemma}
 
\begin{proof}
It is plain that every isomorphism is a $\sG$-cofibration. Let $f\colon A\to B$ and $g\colon B\to C$ be $\sG$-cofibration; then for every $A\to M\xrightarrow{h}N$, if $h\in \sG$ then also the morphism $M\amalg_AB\xrightarrow{h_B}N\amalg_AB$ belongs to $\sG$, and therefore also the morphism 
\[ M\amalg_AC= (M\amalg_AB)\amalg_BC\xrightarrow{h_C}(N\amalg_AB)\amalg_BC=N\amalg_AC\;\]
belongs to $\sG$.
Let $A\to B$ be a $\sG$-cofibration and $A\to C$ a morphism. For every $C\to M\xrightarrow{h\in \sG}N$ we have 
\[ M\amalg_C(C\amalg_AB)=M\amalg_AB\xrightarrow{\,\sG\,}N\amalg_AB = N\amalg_C(C\amalg_AB)\;,\]
and then $C\to C\amalg_AB$ is a $\sG$-cofibration.

Finally, assume that $\sG$ is closed under retracts and consider a retraction 
\[ \xymatrix{A\ar[r]\ar[d]^f&C\ar[r]^{p}\ar[d]^g&A\ar[d]^f\\
B\ar[r]&D\ar[r]^q&B	.}\]
Then every morphism $A\xrightarrow{\alpha} M$ gives a commutative diagram 
\[\xymatrix{M\ar[r]^{\Id}&M\ar[r]^{\Id}&M\\
A\ar[r]\ar[d]^f\ar[u]^\alpha&C\ar[r]^{p}\ar[d]^g\ar[u]^{p\alpha}&A\ar[d]^f\ar[u]^\alpha\\
B\ar[r]&D\ar[r]^q&B}\]
and then a functorial retraction $M\amalg_AB\to M\amalg_CD\to M\amalg_AB$.

If $g$ is a $\sG$-cofibration, then  $f$ is a $\sG$-cofibration, since  
for every $A\to M\xrightarrow{\sG}N$ the morphism 
$M\amalg_AB\to N\amalg_AB$ is a retract of $M\amalg_CD\xrightarrow{\sG} N\amalg_CD$.
\end{proof}


From now on we restrict to consider the case where $\sG=\sW$ is the class of weak equivalences, and we shall denote by $\Cof_{\sW}$  the class of $\sW$-cofibrations.
Moreover, an object in $\bM$ is called $\sW$-\textbf{cofibrant} if its initial map is a $\sW$-cofibration.
Recall \cite[Def. 13.1.1]{Hir03} that a model category is called left-proper if weak equivalences are preserved under pushouts along cofibrations; equivalently, a model category is left-proper if and only if every cofibration is a $\sW$-cofibration ($\sC\subset \Cof_{\sW}$). 

The class $\Cof_{\sW}$ of $\sW$-cofibrations was considered by Grothendieck in his personal approach to model categories \cite[page 8]{derivateur}, and more recently by Batanin and Berger \cite{BB} under the name of $h$-cofibrations.

\begin{lemma}\label{lem.preservaweak}
In a left-proper model category   weak equivalences between $\sW$-cofibrant objects are preserved by pushouts, i.e. for every commutative diagram
\[\begin{matrix}\xymatrix{A\ar[dr]_{g}\ar[r]^{f}&E\ar[d]^{h}\\
&D}\;\end{matrix},\qquad f,g\in \Cof_{\sW},\quad h\in \sW,\]
and every morphism $A\to B$ the pushout map $E\amalg_A B\to D\amalg_A B$ is a weak equivalence.
\end{lemma}

\begin{proof}
Consider a factorization $A\xrightarrow{\alpha}P\xrightarrow{\beta}B$ with $\alpha\in \sC\subset\Cof_{\sW}$, $\beta\in \sW$ and then apply the \emph{2 out of 3} axiom to the diagram 
\[ \begin{matrix}\xymatrix{		E\amalg_A P\ar[d]^{\sW}\ar[r]^{\sW}&E\amalg_A B\ar[d]\\
D\amalg_A P\ar[r]^{\sW}&D\amalg_A B	}\end{matrix} \]
to obtain the statement.
\end{proof}

\begin{corollary}\label{cor.Gcofibration} In a left-proper model category a morphism $f\colon A\to B$ is a $\sW$-cofibration if and only if for every $A\to M\xrightarrow{g}N$ with $g\in \sW\cap \sF$, the pushout morphism $M\amalg_AB\xrightarrow{\quad}N\amalg_AB$ belongs to $\sW$.
\end{corollary}

\begin{proof}
The proof follows from the fact that every weak equivalence is the composition of a trivial cofibration and a trivial fibration, and trivial cofibrations are preserved under pushouts.
\end{proof}

\begin{example}\label{ex.esempiononcofibrante} 
In the model category $\CDGA_{\K}$ of commutative differential graded 
$\K$-algebras consider the polynomial algebras: 
\[ A=\K[x],\qquad B=\K[x,y],\quad \bar{x}=+1,\quad \bar{y}=-1,\quad dy=yx\,.\]
Then $B$ is not cofibrant and the natural inclusion $i\colon A\to B$ is not a $\sW$-cofibration.

1) In order to prove that $B$ is not cofibrant consider the polynomial  algebra 
\[ D=\K[x,y,z],\qquad \bar{z}=0,\quad dy=z,\quad dz=0,\]
together with the surjective morphism 
\[ q\colon D\to B,\qquad q(x)=x,\; q(y)=y,\; q(z)=yx\,.\] 
It is immediate to see that $i$ is a weak equivalence and by K\"{u}nneth formula also the inclusion 
\[ gi\colon A\to D=\K[x]\otimes_{\K}\K[y,z]\]
is a weak equivalence. Hence $q$ is a trivial fibration and then if $B$ is cofibrant there exists a morphism $f\colon B\to D$ such that $qf=\id_B$. Any such $f$ should satisfy
\[ f(x)=xh,\quad f(y)=yk,\qquad h,k\in \K[z],\qquad h,k\not=0,\]
and this gives a contradiction since  
\[ df(y)=d(yh)=zh\not= f(dy)=f(yx)=yxhk\,.\]

2) Consider the retraction of polynomial algebras 
\[ A=\K[x]\xrightarrow{j}\K[x,t]\xrightarrow{q}\K[x],\qquad \bar{t}=0,\quad dt=xt,\quad j(x)=q(x)=x,\quad q(t)=0\,.\] 
Since $\K$ is assumed of characteristic $0$ both $j$ and $q$ are quasi-isomorphisms. 
In order to prove that $i\colon A\to B$ is not a $\sW$-cofibration we shall  prove that the pushoutout of $q$ under $i$ is not a weak equivalence:   in fact 
\[ \K[x,t]\otimes_AB=\K[x,t]\otimes_{\K[x]}\K[x,y]=\K[x,y,t],\qquad dt=xt,\; dy=yx,\]
and the element $yt$ gives a nontrivial cohomology class which is annihilated by the pushout of 
$q$:
\[  \K[x,t]\otimes_AB=\K[x,y,t]\xrightarrow{t\mapsto 0}\K[x]\otimes_{\K[x]}\K[x,y]=\K[x,y]\,.\]
\end{example}

\begin{example}\label{ex.bruttobrutto} The natural inclusion morphism 
\[ A=\frac{\K[\epsilon]}{(\epsilon^2)}\to B=A[x_0,x_1,x_2,\ldots],\qquad \bar{\epsilon}=0,\;\bar{x_i}=i,\quad dx_i=\epsilon x_{i+1},\]
in the category $\CDGA_{\K}$ is not a $\sW$-cofibration. To see this consider the (C-FW) factorization
\[ A\to C=A[u,v]\xrightarrow{\;\epsilon,u,v\mapsto 0\;}\K,\qquad \bar{u}=-1,\; \bar{v}=-2,\; du=\epsilon,\; dv=\epsilon u,\]
and it is easy to see that $C\otimes_A B\to \K\otimes_A B=\K[x_0,x_1,\ldots]$ is not a quasi-isomorphism (for instance $x_0$ does not lift to a cocycle in $C\otimes_A B$).
\end{example}

Every morphism $f\colon A\to B$ in $\mathbf{M}$ induces two functors:
\[ f^{\ast}=-\circ f\colon  \mathbf{M}_B\to \mathbf{M}_A,\qquad (B\to X)\mapsto (A\xrightarrow{f}B\to X),\]
\[ f_{\ast}=-\amalg_A B\colon \mathbf{M}_A\to \mathbf{M}_B,\qquad X\mapsto X\amalg_AB\,.\]

According to the definition of the model structure in the undercategories of $\bM$, a morphism $h$ in $\bM_B$ is a weak equivalence (respectively fibration, cofibration) if and only if $f^{\ast}(h)$ is a weak equivalence (respectively fibration, cofibration), see~\cite[p. 126]{Hir03}.

The functor $f_{\ast}$ preserves cofibrations and trivial cofibrations, and $f$ is a $\sW$-cofibration if and only if $f_{\ast}$ preserves weak equivalences. Given a pushout square 
\[\xymatrix{A\ar[d]^h\ar[r]^f\ar@{}[dr]|(.65){\Big\ulcorner}&B\ar[d]^k\\
C\ar[r]^-{g}&C\amalg_AB}\]
we have the base change formula
\begin{equation}\label{equ.projectionformula}
f_{\ast}h^{\ast}=k^{\ast}g_{\ast}\colon \bM_C\to \bM_B,
\end{equation}
which is equivalent to the canonical isomorphism $D\amalg_A B \cong D\amalg_C(C\amalg_A B)$ for every object $D$ in the category $\bM_C$.

\begin{definition}\label{def.flatmorphism}
A morphism $f\colon A\to B$ in $\bM$ is called \textbf{flat} if the functor $f_{\ast}$ preserves pullback diagrams of trivial fibrations. An object $B\in \mathbf{M}_A$ is called flat (over $A$) if the corresponding  morphism $A\to B$ is flat 
in $\bM$. 
\end{definition}

In a more explicit way, a morphism $A\xrightarrow{\;}B$ in a model category $\bM$ is flat if every commutative square 
\[ \xymatrix{	A\ar[d]\ar[r] & E\ar[d] \\
C\ar[r]^{\sF\sW} & D	} \]
gives a pullback square:  
\[ \xymatrix{	(C\times_DE)\amalg_AB\ar[d]\ar[r] & E\amalg_AB\ar[d] \\
C\amalg_AB\ar[r]^{\sF\sW}&D\amalg_AB		} \]
or, equivalently, if $C{\amalg}_AB\xrightarrow{\;}D{\amalg}_AB$ is a trivial fibration and  the natural map 
\[  (C\times_DE)\amalg_AB\to (C\amalg_AB)\times_{D\amalg_AB}(E\amalg_AB) \]
is an isomorphism.

The notion of flatness is preserved under the passage to undercategories and overcategories. 
In particular, given two maps $K\to A\xrightarrow{f}B$ in $M$, the morphism $f$ is flat in $\bM$ if and only if 
it is flat in $\bM_K$.

The above notion of flatness  is motivated by the example of commutative differential graded algebras: we shall prove in the next section that a morphism 
$A\to B$ in $\DGCA_{\K}^{\le 0}$, with $A=A^0$ concentrated in degree 0, is flat in the sense of
Definition~\ref{def.flatmorphism} if and only it is DG-flat in the sense of \cite{A-F}.

\begin{remark}
Altough the above notion of flatness also makes sense in categories of fibrant objects 
it seems that its utility is restricted to the realm of left-proper model categories. 
It is important to point out that flatness 
is not invariant under weak equivalences, and then it does not make sense to talk about flat morphisms in the homotopy category.
\end{remark}

\begin{lemma}\label{lem.flatWcof}
Every flat morphism is a $\sW$-cofibration. 
\begin{proof} Assume $A\to B$ flat, given $A\to M\xrightarrow{\sW} N$,  consider a factorization 
$A\to M\xrightarrow{\sC\sW}P\xrightarrow{\sF\sW}M$. Then 
\[M\amalg_AB\xrightarrow{\sC\sW}P\amalg_AB=P\amalg_M(M\amalg_A B)\]
is a trivial cofibration by model category axioms, while 
\[P\amalg_AB\xrightarrow{\sF\sW} N\amalg_AB\]
is a trivial fibration by flatness.
\end{proof}
\end{lemma}

\begin{lemma}\label{lem.flatclosedunderretracts}
The class of flat morphisms is stable under composition, pushouts and retractions.
\begin{proof}
Composition: let $A\xrightarrow{f}B\xrightarrow{g}C$ be two flat morphisms, then both 
the functors $f_{\ast}\colon \mathbf{M}_A\to \mathbf{M}_B$ and  
$g_{\ast}\colon \mathbf{M}_B\to \mathbf{M}_C$ preserve pullback diagrams of trivial fibrations. Therefore also 
$(gf)_{\ast}=g_{\ast}f_{\ast}$ preserves pullback diagrams of trivial fibrations.

Pushout: let $A\xrightarrow{f}B$, $A\xrightarrow{\;}C$ be two  morphisms with $f$ flat. Then it follows from the base change Formula~\eqref{equ.projectionformula} that 
$g\colon C\to C\amalg_AB$ is also flat.

Retracts: let $\bC$ be any category, and denote by $\bC^{\Delta^1\times\Delta^1}$ the category of commutative squares in 
$\bC$. It is easy and completely straightforward to see that every retract of a pullback (respectively, pushout) square in 
$\bC^{\Delta^1\times\Delta^1}$ is a pullback (respectively, pushout) square.  
Consider now a retraction \[ \xymatrix{A\ar[r]\ar[d]^f&C\ar[r]^{p}\ar[d]^g&A\ar[d]^f\\
B\ar[r]&D\ar[r]^q&B}\]
in  $\bM$, with $g$ a flat morphism. By the universal property of coproduct, every map $A\to X$ gives a canonical retraction 
\[ X\amalg_AB\to X\amalg_CD\to X\amalg_AB\,.\]
Therefore, every commutative square $\xi\in \bM_A^{\Delta^1\times\Delta^1}$ gives a retraction 
$\xi\amalg_AB\to \xi\amalg_CD\to \xi\amalg_AB$ in the category
$\bM^{\Delta^1\times\Delta^1}$. If $\xi$ is the pullback square of a trivial fibration, then also 
$\xi\amalg_CD$ is the pullback of a trivial fibration. Since trivial fibrations and pullback squares are stable under retracts, it follows that also $\xi\amalg_AB$ is the pullback square of a trivial fibration.
\end{proof}
\end{lemma}

\subsection{The coFrobenius condition}

For the application we have in mind, here and in the forthcoming papers, it is useful to introduce the following definitions.  

\begin{definition}\label{def.SLP} 
We shall say that a model category is \textbf{strong left-proper} if every cofibration is flat.
\end{definition}

\begin{definition}\label{def.frobenius}
A model category satisfies the \textbf{coFrobenius condition} if pushouts along cofibrations preserve trivial fibrations.
\end{definition}

\begin{remark}
1) Every strong left-proper model category satisfies the coFrobenius condition. The converse holds under mild assumptions, see Proposition~\ref{prop.frobenius}.

2) Every model category satisfying the coFrobenius condition is left-proper. The proof is essentially the same as the one of Lemma~\ref{lem.flatWcof}.
The converse is false in general, see Example~\ref{example.Top}.
\end{remark}

The name ``coFrobenius condition" of Definition~\ref{def.frobenius} is due to its dual property, the Frobenius condition, which has been already considered in the literature, \cite{GS17}.

\begin{example}\label{example.Top}
Denote by $\Top$ the category of topological spaces endowed with the standard model structure, \cite{QuillenHA}. It is well-known that $\Top$ is left-proper (see e.g. \cite[Theorem 13.1.10]{Hir03}) but it does not satisfy the coFrobenius condition, in particular it is not strong left-proper.
In order to prove the claim consider the cofibration $\iota\colon 0\to [0,1]$ defined as the natural inclusion, together with the trivial fibration $\pi\colon [0,1]\to 0$. The pushout map $[0,1]\amalg_0[0,1]\to [0,1]$ of $\pi$ along $\iota$ is not a Serre fibration.

Similarly one can prove that the category $\sSet$ of simplicial sets endowed with the Quillen's model structure, \cite{QuillenHA}, does not satisfy the coFrobenius condition, while left-properness immediately follows recalling that all objects are cofibrant.
\end{example}

\begin{proposition}\label{prop.frobenius}
Let $\bM$ be a cofibrantly generated model category where every generating cofibration is flat. Then $\bM$ is strong left-proper if and only if $\bM$ satisfies the coFrobenius condition.
\end{proposition}

\begin{proof}
If $\bM$ is strong left-proper then it satisfies the coFrobenius condition. Conversely, let $I$ be the set of generating cofibrations of $\bM$; by hypothesis every map of $I$ is flat. Recall that every cofibration in $\bM$ is a retract of a transfinite composition of pushouts of maps in $I$,~\cite[Prop. 2.1.18]{Hov99}. Therefore it is sufficient to show that given an ordinal $\underline{\lambda}$ together with a $\underline{\lambda}$-sequence
\[ A_0 \xrightarrow{f_0} A_1\xrightarrow{f_1} \cdots \to A_{\lambda} \xrightarrow{f_{\lambda}} \cdots \]
in $\bM$ where each $f_{\lambda}$ is a flat cofibration for $\lambda<\underline{\lambda}$, then the transfinite composition $f_{\underline{\lambda}}\colon A_0\to \colim\limits_{\lambda<\underline{\lambda}} A_{\lambda}$ is flat. For simplicity of notation we shall denote $A_{\underline{\lambda}}=\colim A_{\lambda}$.

Consider a commutative square
\[ \xymatrix{	A_0\ar@{->}[r]\ar@{->}[d] & C\ar@{->}[d]^{\sF\sW} \\
D\ar@{->}[r] & E 	} \]
with $C\to E$ a trivial fibration. Recall that filtered colimits commute with finite limits, so that we have the following chain of isomorphisms
\[ \begin{aligned}
(D\times_EC)\amalg_{A_0}A_{\underline{\lambda}} &\cong \colim\left( (D\times_EC)\amalg_{A_0}A_{\lambda}\right) \\
& \cong \colim\left( (D\amalg_{A_0}A_{\lambda})\times_{(E\amalg_{A_0}A_{\lambda})}(C\amalg_{A_0}A_{\lambda})  \right) \\
& \cong \colim(D\amalg_{A_0}A_{\lambda})\times_{\colim(E\amalg_{A_0}A_{\lambda})}\colim(C\amalg_{A_0}A_{\lambda}) \\
& \cong (D\amalg_{A_0}A_{\underline{\lambda}})\times_{(E\amalg_{A_0}A_{\underline{\lambda}})}(C\amalg_{A_0}A_{\underline{\lambda}})
\end{aligned} \]
where the second isomorphism follows from the flatness of the maps $A_0\to A_{\lambda}$, $\lambda<\underline{\lambda}$.

We are left with the proof that the morphism $C\amalg_{A_0}A_{\infty}\to E\amalg_{A_0}A_{\underline{\lambda}}$ is a trivial fibration; 
this follows by the coFrobenius condition.
\end{proof}

\begin{example}
By Proposition~\ref{prop.frobenius} it follows that the model category $\CDGA_{\K}^{\le 0}$ is strong left-proper, since 
it satisfies the coFrobenius condition and generating cofibrations are flat.
We shall reprove this fact in Corollary~\ref{cor.cofibrationflat}.
The same argument works also, mutatis mutandis, for proving that $\DGCA_{\K}$ is strong left-proper.

The model category $\mathbf{sAlg}_R$ of simplicial commutative algebras over a commutative ring $R$ (endowed with the model structure defined in \cite[Sec. 4.3]{GS})  is strong left-proper. In fact this category is left-proper, \cite[Lemma 3.1.1]{schwede}, and every cofibration is a retract of a free morphism, \cite[Prop. 4.21]{GS}. The conclusion is now an immediate consequence of the fact that the pushout of commutative simplicial rings is given by degreewise tensor product. 
\end{example}

For future purposes we now prove the following useful result.

\begin{lemma}\label{lem.pushoutFW} Let $\bM$ be a model category satisfying the coFrobenius condition. Assume moreover that for every pair of morphisms $A\to B\to C$, if $A\to C$ is a fibration and $A\to B$ is a trivial fibration,  then  
$B\to C$ is a fibration. 
 
Then trivial fibrations between flat objects are preserved by pushouts.
\end{lemma}

\begin{proof} Given a diagram
\[ \xymatrix{	A\ar[dr]_{\flat}\ar[r]^{\flat} & E\ar[d]^{\sF\sW} \\
 & D	} \]
together with a morphism $A\to B$, consider a factorization $A\xrightarrow{\sC}P\xrightarrow{\sF\sW}B$. By the coFrobenius condition the morphism $E\amalg_AP\to D\amalg_AP$ is a trivial fibration. Moreover, since $A\to E$ and $A\to D$ are flat the morphisms
\[ E\amalg_AP\xrightarrow{\sF\sW}E\amalg_AB,\qquad D\amalg_AP\xrightarrow{\sF\sW}D\amalg_AB \]
are trivial fibrations, so that the commutative diagram 
\[ \begin{matrix}\xymatrix{		E\amalg_AP\ar[r]^{\sF\sW}\ar[d]^{\sF\sW}&E\amalg_AB\ar[d]\\
 D\amalg_AP\ar[r]^{\sF\sW}&D\amalg_AB		}\end{matrix}\]
gives the statement.
\end{proof}

\bigskip
\section{Flatness in $\DGCA_{\K}^{\le 0}$}
\label{sec.flatness}

Unless otherwise stated we shall consider the category $\DGCA_{\K}^{\le 0}$ equipped with the 
projective model structure. Recall that a morphism $A\to B$ in 
$\DGCA_{\K}^{\le 0}$ is a semifree extension if $B=A[\{x_i\}]$ is a polynomial extension 
in an arbitrary number of variables of non-positive degree.

For every differential graded commutative algebra $A$ we shall denote by 
$\DGMod(A)$ (resp.: $\DGMod(A)^{\le 0}$) the category of differential graded modules over $A$ (resp.: concentrated in non-positive degrees). 
For every module 
$M\in \DGMod(A)$ we shall denote by $A\oplus M$ the trivial extension.

For every $A\in \DGCA_{\K}^{\le 0}$ we shall denote by $A\to A[d^{-1}]$ the semifree extension, where $d^{-1}$ has degree $-1$ and $dd^{-1}=1$. For every $A$-module $M$ the tensor product $A[d^{-1}]\otimes_AM$ is isomorphic to the mapping cone $M[1]\oplus M$ of the identity and therefore it is an acyclic $A$-module. For every morphism 
$A\to B$ of algebras we have $A[d^{-1}]\otimes_A B=B[d^{-1}]$ and then for every $A$-module $M$ 
we have 
\[ (A[d^{-1}]\otimes_AM)\otimes_AB\simeq B[d^{-1}]\otimes_B(M\otimes_AB)\,,\]
i.e., mapping cone commutes with tensor products. Finally, the same proof as in the classical case shows that the functor $-\otimes_AB\colon \DGMod(A)^{\le 0}\to \DGMod(B)^{\le 0}$ is right exact, i.e., preserves the class of exact sequences of 
type $M\to N\to P\to 0$.

\begin{lemma}\label{lem.wcofincdga} A morphism $f\colon A\to B$  in $\DGCA_{\K}^{\le 0}$ is a $\sW$-cofibration if and only if 
the graded tensor product $-\otimes_AB\colon \DGMod(A)^{\le 0}\to \DGMod(B)^{\le 0}$ preserves the class of acyclic modules.
\end{lemma}

\begin{proof} The ``only if'' pat is clear  since for every acyclic $A$-module $M$ the natural 
inclusion $A\to A\oplus M$ is a weak equivalence. The ``if'' part is a consequence of the fact that 
the tensor product commutes with mapping cones and the well known fact that a morphism of $A$-modules is a weak equivalence if and only if its mapping cone is acyclic.   
\end{proof}

\begin{theorem}\label{thm.2flatnessCDGA}
Let $f\colon A\to B$ be a morphism in $\DGCA_{\K}^{\le 0}$. The following conditions are equivalent:
\begin{enumerate}
\item  the graded tensor product $-\otimes_AB\colon \DGCA_A^{\le 0}\to \DGCA_B^{\le 0}$ preserves pullback squares of trivial fibrations, i.e., $f$ is flat in the sense of  Definition~\ref{def.flatmorphism};

\item the graded tensor product $-\otimes_AB\colon \DGCA_A^{\le 0}\to \DGCA_B^{\le 0}$ preserves the classes of injections and trivial fibrations;
 
\item the graded tensor product $-\otimes_AB\colon \DGMod(A)^{\le 0}\to \DGMod(B)^{\le 0}$ preserves the class of 
quasi-isomorphisms  and for every short exact sequence 
$0\to M\to N\to P\to 0$ of differential graded $A$-modules, the 
sequence 
\[ 0\to M\otimes_AB\to N\otimes_AB\to P\otimes_AB\to 0\]
is exact.
\end{enumerate}
\end{theorem}

\begin{proof} 
It is clear that (3) implies (2).

We now prove that (1) implies (3). If $M\to N$ is a quasi-isomorphism of $A$-modules, then 
$A\oplus M\to A\oplus N$ is a weak equivalence in $\DGCA_A^{\le 0}$ and, since every flat morphism is a $\sW$-cofibration we also have that 
\[ (A\oplus M)\otimes_A B=B\oplus (M\otimes_AB)\to B\oplus (N\otimes_AB)=(A\oplus N)\otimes_A B\]
is a weak equivalence. Consider now a short exact sequence $0\to M\to N\to P\to 0$ in 
$\DGMod(A)^{\le 0}$. Then we have a pullback square of trivial fibrations
\[ \xymatrix{A\oplus(A[d^{-1}]\otimes_AN)\ar[d]\ar[r]&A\oplus(A[d^{-1}]\otimes_AM)\ar[d]\\
A\ar[r]&A\oplus(A[d^{-1}]\otimes_AP)}\]
and then also 
\[ \xymatrix{B\oplus(B[d^{-1}]\otimes_B(N\otimes_AB))\ar[d]\ar[r]&B\oplus(B[d^{-1}]\otimes_B(M\otimes_AB))\ar[d]\\
B\ar[r]&A\oplus(B[d^{-1}]\otimes_B(P\otimes_AB))}\]
is a pullback square of a trivial fibration: this is possible if and only if the sequence 
\[ 0\to M\otimes_AB\to N\otimes_AB\to P\otimes_AB\to 0\]
is exact. 

Finally we prove that (2) implies (1). By using trivial extensions we immediately see that 
for every injective morphism $M\to N$ of $A$-modules, the induced map
$M\otimes_AB\to N\otimes_AB$ is still injective.

By hypothesis the functor $-\otimes_AB$ preserves the class of trivial fibrations. Then we only need to show that it commutes with pullbacks of a given trivial fibration $f\colon P\xrightarrow{\sF\sW} Q$. To this aim, consider a morphism $C\to Q$ and the pullback $P\times_QC$ represented by the commutative diagram
\[ \xymatrix{	0 \ar@{->}[r] & \ker(f) \ar@{->}[r] \ar@{->}[d]_{\id} & P\times_QC \ar@{->}[r] \ar@{->}[d] & C \ar@{->}[d] \ar@{->}[r] & 0 \\
0 \ar@{->}[r] & \ker(f) \ar@{->}[r] & P \ar@{->}[r] & Q \ar@{->}[r] & 0	} \]
whose rows are exact. Applying the right exact functor $-\otimes_AB$ we obtain the commutative diagram
\[ \xymatrix{	0 \ar@{->}[r] & \ker(f)\otimes_AB \ar@{->}[r] \ar@{->}[d]_{\id} & \left(P\times_QC\right)\otimes_AB \ar@{->}[r] \ar@{->}[d] & C\otimes_AB \ar@{->}[d] \ar@{->}[r] & 0 \\
0 \ar@{->}[r] & \ker(f)\otimes_AB \ar@{->}[r] & P\otimes_AB \ar@{->}[r] & Q\otimes_AB \ar@{->}[r] & 0	} \]
whose rows are exact by hypothesis. It follows that $\left(P\times_QC\right)\otimes_AB$ is (isomorphic to) the pullback $(P\otimes_AB) \times_{(Q\otimes_AB)}(C\otimes_AB)$ as required.
\end{proof}

Notice that in the model category $\DGCA_{\K}^{\le 0}$ not every $\sW$-cofibration is flat: consider for instance 
the morphism of $\K$-algebras $f\colon\K[x]\to \K[d^{-1}]$, $\deg(x)=0$,  $f(x)=0$. 
Notice also that our definition of flatness of a morphism in  $\CDGA_{\K}^{\le 0}$ differs substantially from the notion of flat morphism  given in \cite{TVII}: this will be especially clear after the following two corollaries.

\begin{corollary}\label{cor.degreewiseflat}
Let $f\colon A\to B$ be a morphism in $\DGCA_{\K}^{\le 0}$ and assume that $A$ is concentrated in degree $0$. Then $f$ is flat in the sense of  Definition~\ref{def.flatmorphism} if and only if $B^j$ is a flat $A$-module for every index $j$.
\end{corollary}

\begin{proof} If $f$ is flat, then by condition (3) of Theorem~\ref{thm.2flatnessCDGA} it follows that 
every $B^j$ is a flat $A$-module. 
Conversely, if every $B^j$ is flat then for every short exact sequence 
$0\to M\to N\to P\to 0$ of differential graded $A$-modules, the 
sequence 
\[ 0\to M\otimes_AB\to N\otimes_AB\to P\otimes_AB\to 0\]
is exact. If $M\to N$ is a quasi-isomorphism in $\DGMod(A)^{\le 0}$, 
since both $B,M,N$ are bounded above, for every $j$ the morphism 
$M\otimes_AB^j\to N\otimes_AB^j$ is a quasi-isomorphism of complexes of $A$-modules and   
a standard spectral sequence argument implies that also $M\otimes_AB\to N\otimes_AB$ is a quasi-isomorphism. \end{proof}

\begin{corollary}\label{cor.cofibrationflat}
In the model category $\DGCA_{\K}^{\le 0}$ every cofibration is flat. In particular $\DGCA_{\K}^{\le 0}$ is left-proper.
\end{corollary}

\begin{proof}  The second part of the corollary is well known, nonetheless we give here a sketch of proof of the left-properness for the reader convenience and reference purposes.

Since  every cofibration in $\DGCA_{\K}^{\le 0}$ is a retract  of a semifree extension, according to Lemma~\ref{lem.flatclosedunderretracts} it is sufficient to prove that every semifree extension $A\to B$ is flat. We use Theorem~\ref{thm.2flatnessCDGA} and we prove that  
$-\otimes_AB\colon \DGCA_A^{\le 0}\to \DGCA_B^{\le 0}$ preserves the classes of injections and trivial fibrations.

The preservation of the class of injections is clear since the injectivity of a morphism is independent of the differentials and, as a graded module, $B$ is isomorphic to a direct sum of copies of $A$.

We have already seen that tensor product preserves the class of surjective morphisms and in order to conclude the proof we need to show that the semifree extension $A\to B$ is a $\sW$-cofibration.

Write $B=A[x_i]$, $i\in I$, and notice that for every  finite subset $U\subset B$ there exists a finite subset of indices $J\subset I$ such that $A[x_j]$, $j\in J$, is a differential graded subalgebra of $B$ containing $U$. Thus it not restrictive to assume that $B$ is a finitely generated semifree $A$-algebra.
Finally, since $\sW$-cofibrations are stable under finite composition we can reduce to the case 
$B=A[x]$, with $\bar{x}\le 0$ and $dx\in A$. 

Denoting by $B_n\subset B$, $n\ge 0$, the differential graded $A$-submodule of polynomial of degree $\le n$ in $x$, for every morphism $A\to C$ the cohomology of $C\otimes_A B$ can be computed via the spectral sequence associated to the filtration $C_n=C\otimes_A B_n$, whose  first page is a direct sum of copies of the cohomology of $C$. This clearly implies that the free simple extension $A\to B=A[x]$ is a 
$\sW$-cofibration.
\end{proof}

The following result is the analog (of the Artin version, cf. \cite[Lemma A.4, item (a)]{Sernesi}) of Nakayama's lemma in the category 
$\DGCA_{\K}^{\le 0}$.

\begin{proposition}\label{prop.naka} 
Let $I$ be a nilpotent differential graded ideal of an algebra $A\in \DGCA_{\K}^{\le 0}$ and 
let $f\colon P\to Q$ be a morphism of flat commutative differential graded $A$-algebras. Then $f$ is an isomorphism (resp.: a weak equivalence) if and only if the induced morphism 
\[ \bar{f}\colon P\otimes_A\frac{A}{I}=\frac{P}{IP}\longrightarrow Q\otimes_A\frac{A}{I}=\frac{Q}{IQ}\]
is an isomorphism (resp.: a weak equivalence).
\end{proposition}  

\begin{proof} Denoting by $B=A/I$, 
it is not restrictive to assume that $I$ is a square zero ideal; in particular $I$ is a 
$B$-module
and we have a short exact sequence of $A$-modules
\[ 0\to I\to A\to \frac{A}{I}=B\to 0\,.\]
By Theorem~\ref{thm.2flatnessCDGA} we get a morphism of two short exact sequences of $A$-modules 
\begin{equation}\label{equ.naka} 
\begin{matrix}\xymatrix{	0 \ar@{->}[r] & P\otimes_AI\ar@{->}[r] \ar@{->}[d]^{g} &P \ar@{->}[r] \ar@{->}[d]^f & P\otimes_AB \ar@{->}[d]^{\bar{f}} \ar@{->}[r] & 0 \\
0 \ar@{->}[r] & Q\otimes_AI \ar@{->}[r] & Q \ar@{->}[r] & Q\otimes_AB \ar@{->}[r] & 0	}\end{matrix}
\end{equation}
where 
\[ g=\bar{f}\otimes\id_{I}\colon (P\otimes_AB)\otimes_B I\to (Q\otimes_AB)\otimes_B I\,.\]
If $\bar{f}$ is an isomorphism, then also $g$ is an isomorphism and the conclusion follows by snake lemma.
If $\bar{f}$ is a quasi-isomorphism, then it is a weak equivalence of flat $B$-algebras and then also 
$g$ is a weak equivalence by Lemma~\ref{lem.preservaweak}.
The proof now follows immediately by the five lemma applied to the long cohomology exact sequence of
\eqref{equ.naka}. 
\end{proof}

We shall denote by $\DGArt_{\K}^{\le 0}\subset 
\DGCA_{\K}^{\le 0}$ the full subcategory of differential graded local Artin algebra with residue field $\K$. By definition a commutative differential graded algebra $A=\oplus A^i$ belongs to 
$\DGArt_{\K}^{\le 0}$ if $A^0$ is a local Artin algebra with maximal ideal $\mathfrak{m}_{A^0}$ such that the composition $\alpha\colon \K\to A^0\to {A^0}/{\mathfrak{m}_{A^0}}$ is an isomorphism, and 
$A$ is a finitely generated  graded $A^0$-module. In particular $A$ is a finite dimensional differential graded $\K$-vector space and
$\mathfrak{m}_{A}:=\mathfrak{m}_{A^0}\oplus A^{<0}$ is a nilpotent differential graded ideal. For simplicity of notation we always identify $\K$ with the residue field $A/\mathfrak{m}_A$ via the isomorphism $\alpha$.
The following result is an immediate consequence of Proposition~\ref{prop.naka}.

\begin{corollary}\label{cor.naka} 
Let  $f\colon P\to Q$ be a morphism of flat commutative differential graded $A$-algebras, with 
$A\in \DGArt_{\K}^{\le 0}$.
Then $f$ is an isomorphism (resp.: a weak equivalence) if and only if the induced morphism 
\[ \bar{f}\colon P\otimes_A\K\longrightarrow Q\otimes_A\K\]
is an isomorphism (resp.: a weak equivalence).
\end{corollary}

We denote by $\Art_{\K}\subset \DGArt_{\K}^{\le 0}$ the full subcategory of local Artin algebras with residue field $\K$, i.e., 
$A\in \Art_{\K}$ if and only if it is concentrated in degree $0$ and $A\in \DGArt_{\K}^{\le 0}$. The following 
corollary, equivalent to \cite[Lemma 5.1.1]{H04}, is a reformulation of the classical meaning of flatness in terms of relations \cite[Prop 3.1]{Art}, \cite[Thm. A.10]{Sernesi}.

\begin{corollary}\label{cor.flatnessandrelation} 
Let  $R$ be a flat commutative differential graded $A$-algebra, with 
$A\in \Art_{\K}$. Then the natural map $R\to H^0(R)$ is a trivial fibration of flat $A$-algebras if and only if 
$R\otimes_A\K\to H^0(R\otimes_A\K)$ is a trivial fibration.
\end{corollary}

\begin{proof} As already said this is an easy consequence of Corollary~\ref{cor.degreewiseflat} and standard fact about flatness and we give a direct proof only fon completeness of exposition.  
Since $H^0(R\otimes_A\K)=H^0(R)\otimes_A\K$, 
one implication follows immediately from Lemma~\ref{lem.pushoutFW}. 
Conversely, if $H^i(R\otimes_A\K)=0$ for every $i<0$, we can prove by induction on the length of $A$ that 
also $H^i(R)=0$ for every $i<0$. In fact, since $R$ is flat over $A$, every  small extension
\[ 0 \to \K \to A\to B \to 0 \]
of Artin rings gives a short exact sequence 
\[ 0 \to R\otimes_A\K \to R\to R\otimes_AB \to 0 \]
with $R\otimes_AB$ flat over $B$ and the conclusion follows by the cohomology long exact sequence.
According to Corollary~\ref{cor.degreewiseflat} the morphism $R\to H^0(R)$ is a flat resolution of the $A$-module $H^0(R)$, therefore  
\[ \Tor_1^A(H^0(R),\K)=H^{-1}(R\otimes_A\K)=0\]
and $H^0(R)$ is flat over $A$.
\end{proof}

\bigskip

\section{Deformations of a morphism}\label{section.deformationmorphism}

In order to make a ``good''  deformation theory of a morphism in a model category, we need to introduce 
a class of morphisms  that heuristically corresponds to extensions for which 
Corollary~\ref{cor.naka} is valid in an abstract setting. 

\begin{definition}\label{def.M(K)}
Let $\bM$ be a left-proper model category. For every object $K\in \bM$ we 
denote by $\bM(K)$ the full subcategory of $\bM\!\downarrow\! K$ whose objects are the morphisms $A\to K$ that have the following property: for every commutative diagram  
\[ \xymatrix{	A \ar@/_1pc/[dr]^{\flat} \ar@{->}[r]^{\flat} & E\ar[d]^{h} \\
 & D	} \]
the morphism $h$ is a weak equivalence (respectively: an isomorphism) if and only if the induced pushout map $E\amalg_AK\to D\amalg_AK$ is a weak equivalence (respectively: an isomorphism).
\end{definition}

\begin{definition}\label{def.smallextension}
Let $\bM$ be a left-proper model category. A \textbf{small extension} in $\bM$ is a morphism $A\to K$ in $\bM(K)$ for some object $K\in\bM$. The class of small extensions is denoted by $\SExt$.
\end{definition}

\begin{definition}\label{def.deformationXL} 
Let $K\xrightarrow{f}X$ be a morphism in a left-proper model category $\bM$, with $X$ a fibrant object.   
A deformation of $f$ over  $(A\xrightarrow{p}K)\in \bM(K)$ is a commutative diagram 
\[ \xymatrix{	A\ar[d]^p\ar[r]^{f_A} & X_A\ar[d] \\
K\ar[r]^f & X	} \]
such that $f_A$ is flat and the induced map $X_A\amalg_AK\to X$ is a weak equivalence. 

A direct equivalence is given by a commutative diagram 
\[ \xymatrix{	A\ar[d]_{g_A}\ar[r]^{f_A} & X_A\ar[d]\\
Y_A\ar[r]\ar[ru]^h & X	} \]
Two deformations are \textbf{equivalent} if they are so under the equivalence relation generated by direct equivalences.
\end{definition} 

Notice that the assumption $(A\xrightarrow{p}K)\in \bM(K)$ implies that the morphism $h$ in Definition~\ref{def.deformationXL} is a weak equivalence. In fact, the pushout along $p$ gives a commutative diagram
\[ \xymatrix{	K\ar[d]_{g_A'}\ar[r]^{f_A'} & X_A\amalg_AK\ar[d] \\
Y_A\amalg_AK\ar[r]\ar[ru]^{h'} & X	} \]
and $h'$ is a weak equivalence by the \emph{2 out of 3} axiom.

We denote either by $\Def_f(A\xrightarrow{p}K)$ or, with a little abuse of notation, by $\Def_f(A)$ the quotient class of deformations up to equivalence.

If every cofibration is flat we can also consider $c$-deformations, defined as in Definition~\ref{def.deformationXL} by replacing flat morphisms with cofibrations. 
We denote by $c\Def_f(A)$ the quotient class of $c$-deformations up to equivalence.

Since flat morphisms and cofibrations are $\sW$-cofibrations (see Lemma~\ref{lem.flatWcof}) according to Lemma~\ref{lem.preservaweak} every morphism $A\to B$ in $\bM(K)$ induces two maps  
\[ \Def_f(A)\to \Def_f(B),\quad c\Def_f(A)\to c\Def_f(B), \qquad X_A\mapsto X_A\amalg_AB\,. \]

\begin{lemma}\label{lemma.defVScdef}
In the above setup, if every cofibration is flat  then: 
\begin{enumerate}

\item the natural morphism $c\Def_f(A)\to \Def_f(A)$ is bijective,

\item for every weak equivalence $A\to B$ in $\bM(K)$  the map 
$\Def_f(A)\to \Def_f(B)$ is bijective.
\end{enumerate}
\end{lemma}
  
\begin{proof} 
1) Replacing every deformation $A\xrightarrow{\flat} X_A$ with a factorization $A\xrightarrow{\sC}\widetilde{X_A}\xrightarrow{\sF\sW}X_A$, by Lemma~\ref{lem.preservaweak} we have $\widetilde{X_A}\amalg_AK\xrightarrow{\sW}X_A\amalg_AK$, and this proves that $c\Def_f(A)\to \Def_f(A)$  is surjective. The injectivity is clear since we can always assume $\widetilde{X_A}=X_A$ whenever $A\to X_A$ is a cofibration, and every direct equivalence of deformations 
\[ \xymatrix{	A\ar[d]_{g_A}\ar[r]^{f_A} & X_A\ar[d] \\
Y_A\ar[r]\ar[ru]^{\sW} & X	} \]
lifts to a diagram 
\[ \xymatrix{	A\ar[d]_{\sC}\ar[r]^{\sC} & \widetilde{X_A}\ar[r]^{\sF\sW} & X_A\ar[d] \\
\widetilde{Y_A}\ar[r]^{\sF\sW}\ar[ru]^{\sW} & Y_A\ar[r]\ar[ru]^{\sW}&X	} \]

2) By the first part we may prove that if $q\colon A\to B$ is a weak equivalence  then 
$c\Def_f(A)\to c\Def_f(B)$ is bijective. For every  c-deformation 
$B\to X_B\to X$, taking a factorization 
\[ \xymatrix{A\ar[d]^q\ar[r]^{\sC}&X_A\ar[d]^{\sF\sW}\\
B\ar[r]&X_B}\]
since weak equivalences are preserved under pushouts along cofibrations we get 
\[ \xymatrix{A\ar[d]^q\ar[r]^{\sC}&X_A\ar[d]^{\sW}\ar[dr]^{\sF\sW}&\\
B\ar[r]&X_A\amalg_AB\ar[r]^{\alpha}&X_B}\]
where $\alpha$ is a weak equivalence of flat $B$-objects: this proves the surjectivity of 
$c\Def_f(A)\to c\Def_f(B)$. 

By the lifting property it is immediate to see that if two $c$-deformations $B\to X_B\to X$ and 
$B\to Y_B\to X$ are directly equivalent, then also every pair of factorizations 
\[ A\xrightarrow{\sC}X_A \xrightarrow{\sW\sF}X_B\to X,\qquad 
A\xrightarrow{\sC}Y_A \xrightarrow{\sW\sF}Y_B\to X\]
gives directly equivalent c-deformations over $A$. The injectivity of 
$c\Def_f(A)\to c\Def_f(B)$ is now clear since for every c-deformation 
$A\xrightarrow{\sC}X_A\to X$ and every factorization 
\[A\xrightarrow{\sC}X_A\xrightarrow{\sC\sW}X_A'\xrightarrow{\sF\sW}X_A\amalg_AB\to X\,,\]
the deformation $A\to X_A\to X$ is equivalent to $A\to X_A'\to X$.
\end{proof}

Thus in a strong left-proper model category we have $c\Def_f=\Def_f$.

\begin{lemma}\label{lem.pushoutcfdefo}
In a strong left-proper model category consider a commutative diagram   
\begin{equation}\label{equ.cbarpushout} 
\begin{matrix}\xymatrix{	 & A\ar[dl]\ar[d]\ar[rd] & \\
X_A\ar[dr] & Z_A\ar[l]_{\sC\sW}\ar[r]^{\sW}\ar[d] & Y_A\ar[dl] \\
 & X & 	}\end{matrix}
\end{equation}
of $c$-deformations $A\to X_A\to X$, $A\to Y_A\to X$ and $A\to Z_A\to X$. Then $A\to X_A\amalg_{Z_A}Y_A\to X$ is a $c$-deformation.
\end{lemma}

\begin{proof}
Since the composite map $A\xrightarrow{\sC} Y_A\xrightarrow{\sC\sW} X_A\amalg_{Z_A}Y_A$ is a cofibration we only need to prove that 
\[ (X_A\amalg_{Z_A}Y_A)\amalg_AK\to X \]
is a weak equivalence. Since $Y_A\to X_A\amalg_{Z_A}Y_A$ is a weak equivalence between flat $A$-objects, looking at the commutative diagram  
\[ \xymatrix{	Y_A\amalg_AK\ar[r]^-{\sW}\ar[d]^{\sW} & (X_A\amalg_{Z_A}Y_A)\amalg_AK\ar[dl] \\
X & 	} \]
the statement follows from the 2 out of 3 property.
\end{proof}

\begin{proposition}\label{prop.pushoutcfdefo}
In a strong left-proper model category two $c$-deformations $A\to X_A\to X$ and $A\to Y_A\to X$ are equivalent if and only if there exists a $c$-deformation $A\to Z_A\to X$ and a commutative diagram 
\begin{equation}\label{equ.cbarequivalence} 
\begin{matrix}\xymatrix{	 & A\ar[dl]\ar[d]\ar[rd] & \\
X_A\ar[r]^{\sC\sW}\ar[dr] & Z_A\ar[d] & Y_A\ar[l]_{\sC\sW}\ar[dl] \\
 & X & 	}\end{matrix}
\end{equation}
\end{proposition}

\begin{proof}
We need to prove that:

1) the relation $\sim$ defined by diagram \eqref{equ.cbarequivalence} is an equivalence relation. This follows immediately from Lemma~\ref{lem.pushoutcfdefo}.

2) if 
\[ \xymatrix{	 & A\ar[dl]\ar[rd] & \\
X_A\ar[rr]^{\sW}\ar[dr] & & Y_A\ar[dl] \\
 & X & 	} \]
is a direct equivalence of $c$-deformations, then $X_A\sim Y_A$. To this end consider a factorization
\[ \xymatrix{	X_A\ar[rdd]_{\sW}\ar[r]^-{\sC} & X_A\amalg_AY_A\ar[d]^{\sC} & Y_A\ar[l]_-{\sC}\ar@(d,r)[ldd]^{\Id} \\
 & Z_A\ar[d]^{\sF\sW} & \\
 & Y_A & 	} \]
and   the morphism $Z_A\amalg_AK\to Y_A\amalg_AK$ is a weak equivalence.
\end{proof} 

\begin{remark}
In the diagram \eqref{equ.cbarequivalence} it is not restrictive to assume that $X_A\amalg_AY_A\to Z_A$ is a cofibration: in fact we can always consider a factorization $X_A\amalg_AY_A\xrightarrow{\sC}Q_A\xrightarrow{\sF\sW}Z_A$ and  the map $Q_A\amalg_AK\to Z_A\amalg_AK$ is a weak equivalence.
\end{remark}

\bigskip
\section{Homotopy invariance of deformations}\label{section.homotopyinvariance}

The aim of this section is to prove that the deformation theory of fibrant objects is invariant under weak equivalences.

The following preliminary technical result is essentially contained in \cite{kenbrown,QuillenHA}. 

\begin{lemma}[Pullback of path objects]\label{lem.pullbackpathobjects}
Let $h\colon Q\to X$ be a fibration of fibrant objects in a model category and let 
\[ X\xrightarrow{\;i\;}X^I\xrightarrow{\;p=(p_1,p_2)\;}X\times X,\qquad  p_1i=p_2i=\id_X,\; i\in \sW,\; p\in \sF,\] 
be a path object of $X$. Then the morphism 
\[ \alpha=(\id_Q,ih,\id_Q)\colon Q\to Q\times_X X^I\times_X Q=
\lim\begin{pmatrix}\xymatrix{Q\ar[dr]^h&&X^I\ar[dl]_{p_1}\ar[dr]^{p_2}&&Q\ar[dl]_h\\
&X&&X&}\end{pmatrix}\]
extends to a  commutative diagram 
\[ \xymatrix{	Q\ar[dd]^h\ar[r]^-{\sC\sW}\ar[dr]^{\alpha} & Q^I\ar[d]^{\beta}\ar[r]^-{\sF\sW} & Q\times_X X^I\ar[d]^{\pi_1} \\
 & Q\times_X X^I\times_X Q\ar[d]^{\pi_2}\ar[r]^-{\pi_1}\ar[ur]^{\gamma} & Q\ar[d]^h \\
X\ar[r]_i & X^I\ar[r]_{p_1} & X	} \]
where:  $Q\times_X X^I$ is the fibered product of $h$ and $p_1$; $\gamma$ is the natural projection on the 
first two factors;  every 
$\pi_i$ denotes the projection on the $i$-th factor.
\end{lemma}

\begin{proof}
Define $Q^I$ by taking a factorization of $\alpha$ as the composition of a trivial cofibration and a fibration $\beta\colon Q^I\to Q\times_X X^I\times_X Q$. Now we have a pullback diagram 
\[ \xymatrix{	Q\times_X X^I\times_X Q\ar[d]_{\gamma}\ar[r]^-{\pi_3} & Q\ar[d]^h \\
Q\times_X X^I\ar[r]_-{p_2\pi_2} & X	} \]
and, since $f$ is a fibration, also $\gamma$ and the composition 
$\gamma\beta\colon Q^I\to Q\times_X X^I$ are fibrations. Finally, the projection $Q\times_XX^I\xrightarrow{\pi_1}Q$ is a weak equivalence since it is the pullback of the trivial fibration $p_1$. Hence $\gamma\beta$ is a weak equivalence by the \emph{2 out of 3} axiom. 
\end{proof}

\begin{lemma}\label{lem.homotopyinjectivity} Let $\tau\colon X\to Y$ be a trivial fibration  of fibrant 
objects in  
a model category $\bM$, and let 
\begin{equation}\label{equ.hinj1} 
\begin{matrix}\xymatrix{A\ar[d]^p\ar[r]^{f_A}&Q\ar[d]^h\\
K\ar[r]^f&X}\end{matrix}\end{equation}
be a $c$-deformation of a morphism $f\colon K\to X$ along $(A\xrightarrow{p}K)\in \bM(K)$.
Then for every morphism $k\colon Q\to X$ such that $\tau h=\tau k$, $kf_A=fp$, the diagram 
\begin{equation}\label{equ.hinj2} 
\begin{matrix}\xymatrix{	A\ar[d]^p\ar[r]^{f_A} & Q\ar[d]^k \\
K\ar[r]^f & X	}\end{matrix}
\end{equation}
is a $c$-deformation equivalent to the previous one.
\end{lemma}

\begin{proof} We have a diagram  
\[ \xymatrix{	 & & & X\ar[dr]^\tau & \\
A\ar[r]^{f_A} & Q\ar[rru]^h\ar[rrd]_k\ar[r] & Q\amalg_AK\ar[ru]_{h'}\ar[rd]^{k'} & & Y \\
 & & & X\ar[ur]_\tau & 	} \]
and by the 2 out of 3 property $k'$ is a weak equivalence, i.e., the square \eqref{equ.hinj2}  is a $c$-deformation:  we need to prove that it is equivalent to  \eqref{equ.hinj1}.

Taking possibly a (CW,F)-factorization of $h$, followed by an extension of $k$:
\[\xymatrix{Q\ar[rr]^k\ar[d]^{\sC\sW}\ar@(ld,ul)[dd]_{h}&&X\ar[dd]^{\tau}\\
Q'\ar[d]^{\sF}\ar[urr]&&\\
X\ar[rr]^{\tau}&&Y}\]
it is not restrictive to assume that $h$ is a fibration.
Since $\tau h=\tau k$ and $\tau$ is a weak equivalence, the maps 
$h$ and $k$ are the same map in the homotopy category $\operatorname{Ho}(\bM_A)$. 
Thus, since $A\to Q$ is a cofibration, the maps $h$ and $k$ are right homotopic: in other words there exists a path object $X\to X^I\xrightarrow{(p_1,p_2)}X\times X$ and a morphism $\phi\colon Q\to X^I$ such that $h=p_1\phi$, $k=p_2\phi$. Taking the pullback of $p_1$ along $h$ we get  the following commutative diagram in $\mathbf{M}_A$:
\[ \begin{matrix}\xymatrix{Q\ar[rr]^-{\psi}\ar[rrd]_\phi\ar@(ur,ul)[rrr]^{\id_Q} && Q\times_XX^I\ar[r]\ar[d] & Q\ar[d]^h \\
 && X^I\ar[r]_{p_1} & X	}\end{matrix}\qquad. \]
Applying Lemma~\ref{lem.pullbackpathobjects} to the fibration $h$, we obtain the commutative diagram 
\[ \xymatrix{	Q^I\ar[d]^{\beta}\ar[r]^-{\sF\sW\;} & Q\times_X X^I\ar[d]^{\pi_1} \\
Q\times_X X^I\times_X Q\ar[d]^{\pi_2}\ar[r]^-{\pi_1}\ar[ur]^{\gamma} & Q\ar[d]^h \\
X^I\ar[r]_{p_1} & X	} \]
and, since $Q$ is cofibrant, the morphism $\psi$ lifts to a morphism $\psi'\colon Q\to Q^I$. Therefore  we have a commutative diagram 
\[ \xymatrix{Q\ar@(r,ul)[rrd]^{\psi}\ar@(d,ul)[ddr]_{\beta\psi'=(\Id,\phi,\eta)}\ar[dr]^{\psi'}&&\\	
&Q^I\ar[d]^{\beta}\ar[r]^-{\sF\sW\;} & Q\times_X X^I\ar[d]^{\pi_1} \\
&Q\times_X X^I\times_X Q\ar[ur]^{\gamma}\ar[r]^-{\pi_1} &Q} \]
In particular $h\eta=p_2\phi=k$, and the morphism $\eta$ gives the required equivalence of deformations:
\[ \xymatrix{	 & Q\ar[dr]^h & \\
A\ar[ru]^{f_A}\ar[rd]_{f_A} & & X \\
 & Q\ar[ur]_k\ar[uu]^\eta & 	} \]
\end{proof}

We are now ready to prove the main result of this section.

\begin{theorem}[Homotopy invariance of deformations]\label{thm.homotopyequivalence}
Let $K\xrightarrow{f}X\xrightarrow{\tau}Y$ be morphisms in a model category $\bM$ and consider a map  
$A\to K$ in $\bM(K)$. If every cofibration is flat and 
$\tau$ is a weak equivalence of fibrant objects, then  the natural  map 
\[ \Def_f(A)\to \Def_{\tau f}(A),\qquad (A\to X_A\to X)\mapsto (A\to X_A\to X\xrightarrow{\tau}Y),\] 
is bijective. 
\end{theorem}

\begin{proof} 
By Ken Brown's lemma it is not restrictive to assume that $\tau\colon X\to Y$ is a trivial fibration of fibrant objects. According to Lemma~\ref{lemma.defVScdef}
 we may replace $\Def(A)$ with  $c\Def(A)$  at any time.

In order to show the surjectivity of $c\Def_f(A)\to c\Def_{\tau f}(A)$ observe that if $A\to Y_A\xrightarrow{h} Y$ is a $c$-deformation, then $K\to Y_A\amalg_A K$ is a cofibration. Therefore the weak equivalence $Y_A\amalg_A K\xrightarrow{h'} Y$ lifts to a weak equivalence $Y_A\amalg_A K\to X$.

Next we prove the injectivity of $c\Def_f(A)\to c\Def_{\tau f}(A)$, i.e.,  that 
two $c$-deformations of $f$, $A\to X_A\to X$ and $A\to Z_A\to X$,  are equivalent in $c\Def_f(A)$ if 
 $A\to X_A\to X\to Y$ and  $A\to Z_A\to X\to Y$ are equivalent in 
$c\Def_{\tau f}(A)$. By the  argument used in the proof of the surjectivity it is not restrictive to assume that $A\to X_A\to X\to Y$ and  $A\to Z_A\to X\to Y$ are  are direct equivalent,  i.e., that there exists a commutative diagram 
\[ \xymatrix{&X_A\ar[r]^k\ar[dd]^{\eta}&X\ar[dr]^{\tau}&\\
A\ar[ur]^{\sC}\ar[rd]_{\sC}&&&Y\\
&Z_A\ar[r]^h&X\ar[ur]^{\tau}&}\]
Now $h\eta\colon X_A\to X$ is clearly equivalent to $h\colon Z_A\to X$, while $k,h\eta\colon X_A\to X$ are equivalent by Lemma~\ref{lem.homotopyinjectivity}. 
\end{proof}

\begin{remark} By Theorem~\ref{thm.homotopyequivalence} it makes sense to define 
deformations of a morphism $K\to X$ even if $X$ is not fibrant in $\bM_K$. To this end 
it is sufficient to consider a fibrant replacement $X\xrightarrow{\sW}Y\xrightarrow{\sF}*$ and define $\Def_X=\Def_Y$. 
\end{remark}

\bigskip

\section{Lifting problems}
\label{sec.lifting}

Let $\bM$ be a strong left-proper model category (i.e. a left-proper model category where every cofibration is flat). The full subcategory of flat objects $^\flat\bM$ inherits the model structure of $\bM$, meaning that $^\flat\bM$ is closed with respect to every axiom even if it may not be complete and cocomplete; for the axioms of a model structure we refer to~\cite{Hov99}. For every morphism $f\colon A\to B$ in $\bM$ the pushout $-\amalg_AB$ defines a functor between the undercategories
\[ f_{\ast}\colon ^\flat\bM_A \to \mbox{ }^\flat\bM_B \]
endowed with the model structures induced by $\bM$. Notice that in general $^\flat(\bM_A)\neq(^\flat\bM)_A$; throughout all the paper we shall denote the category $^\flat(\bM_A)$ of $A$-flat objects simply by $^\flat\bM_A$ as above. By assumption $f_{\ast}$ preserves cofibrations and weak equivalences. Therefore, whenever $f_{\ast}$ preserves fibrations, it makes sense to study whether the following \emph{lifting problems} admit solutions.
\begin{itemize}
\item \textbf{Lifting}: Consider a commutative diagram of solid arrows
\[ \xymatrix{	P_A \ar@{->}[dr]_{g_A} \ar@{->}[dd] \ar@{->}[rr] & & S_A \ar@{-}[d] \ar@{->}[dr]^{p_A} & \\
 & Q_A \ar@{->}[dd] \ar@{->}[rr] \ar@{-->}[ur]|-{h_A} & \ar@{->}[d] & R_A \ar@{->}[dd] \\
P_B \ar@{->}[dr]_{g_B} \ar@{-}[r] & \ar@{->}[r] & S_B \ar@{->}[dr]^{p_B} & \\
 & Q_B \ar@{->}[rr] \ar@{->}[ur]|-{h_B} & & R_B	} \]
in $\bM_A$, where the upper square is in $^\flat\bM_A$ and reduces to the bottom square applying $f_{\ast}$, and moreover the map $g_A$ is a cofibration (respectively: trivial cofibration) and the map $p_A$ is a trivial fibration (respectively: fibration). Then there exists a (dashed) lifting $h_A\colon Q_A\to S_A$ which reduces to $h_B$.
\item \textbf{(CW,F)-factorization}: Given a morphism $g\colon M\to N$ in $^\flat\bM_A$, together with a factorization $M\amalg_AB\xrightarrow{\sC\sW} Q_B\xrightarrow{\sF} N\amalg_AB$ of the map $f_{\ast}(g)=g\amalg_AB$ in $\bM_B$, then there exists a commutative diagram
\[ \xymatrix{	M\ar@{->}@/^2pc/[rr]^{g}\ar@{->}[d]\ar@{-->}[r]_{\sC\sW} & Q_A\ar@{-->}[d]\ar@{-->}[r]_{\sF} & N\ar@{->}[d] \\
M\amalg_AB\ar@{->}[r] & Q_B\ar@{->}[r] & N\amalg_AB 		} \]
in $\bM_A$, where the lower row is obtained by applying the functor $f_{\ast}$ to the upper row.
\item \textbf{(C,FW)-factorization}: Given a morphism $g\colon M\to N$ in $^\flat\bM_A$, together with a factorization $M\amalg_AB\xrightarrow{\sC} Q_B\xrightarrow{\sF\sW} N\amalg_AB$ of the map $f_{\ast}(g)=g\amalg_AB$ in $\bM_B$, then there exists a commutative diagram
\[ \xymatrix{	M\ar@{->}@/^2pc/[rr]^{g}\ar@{->}[d]\ar@{-->}[r]_{\sC} & Q_A\ar@{-->}[d]\ar@{-->}[r]_{\sF\sW} & N\ar@{->}[d] \\
M\amalg_AB\ar@{->}[r] & Q_B\ar@{->}[r] & N\amalg_AB 		} \]
in $\bM_A$, where the lower row is obtained by applying the functor $f_{\ast}$ to the upper row.
\item \textbf{Weak retractions of cofibrations}: Let $g_A\colon P_A\to R_A$ be a cofibration in $^\flat\bM_A$, and consider the diagram of solid arrows
\[ \xymatrix{	Q_A \ar@{-->}[dr]_{\bar{g}_A} \ar@{->}[dd] \ar@{->}[rr]^{j_A} & & P_A \ar@{-}[d] \ar@{->}[dr]^{g_A} \ar@{->}[rr]^{q_A} & & Q_A  \ar@{-}[d] \ar@{-->}[dr]^{\bar{g}_A} & \\
 & S_A \ar@{-->}[dd] \ar@{-->}[rr]^<<<<<<{i_A} & \ar@{->}[d] & R_A \ar@{->}[dd] \ar@{-->}[rr]^<<<<<<{p_A}  & \ar@{->}[d] & S_A \ar@{-->}[dd] \\
Q_B \ar@{->}[dr]_{\bar{g}_B} \ar@{-}[r]^{j_B} & \ar@{->}[r] & P_B \ar@{->}[dr]^{g_B} \ar@{-}[r]^{q_B} & \ar@{->}[r] & Q_B \ar@{->}[dr]^{\bar{g}_B} & \\
 & S_B \ar@{->}[rr]^{i_B} & & R_B \ar@{->}[rr]^{p_A} & & S_B 		} \]
in $\bM_A$, where the bottom rectangle is a retraction of the map $g_B=f_{\ast}(g_A)$ in $\bM_B$, all the horizontal arrows are weak equivalences and the retraction $Q_A\xrightarrow{j_A} P_A\xrightarrow{q_A} Q_A$ reduces to $Q_B\xrightarrow{j_B} P_B\xrightarrow{q_B} Q_B$ applying $f_{\ast}$. Then there exist dashed morphisms giving a retraction of $g_A$ fitting the diagram above, where again all the horizontal arrows are weak equivalences.
\end{itemize}

\begin{remark}[(trivial) cofibrations]
If the \emph{(CW,F)-factorization} lifting problem is satisfied, then the following lifting problem is so.
Given an object $M$ in $^\flat\bM_A$ together with a (trivial) cofibration $g_B\colon M\amalg_AB\to N_B$ with $N_B$ fibrant in $\bM_B$, then there exists a commutative square
\[ \xymatrix{	M\ar@{->}[d]\ar@{-->}[r]^{g_A} & N_A\ar@{-->}[d] \\
M\amalg_AB\ar@{->}[r]^<<<<{g_B} & N_B 	} \]
in $\bM_A$, where $g_A$ is a (trivial) cofibration and $g_B=f_{\ast}(g_A)$.
\end{remark}

Notice that for every surjective map $f$ in $\CDGA_{\K}^{\le 0}$ the functor $f_{\ast}$ preserves fibrations. Motivated by geometric applications in Deformation Theory, the aim of the following subsections is to prove that given a surjective morphism $f\colon A\to B$ in $\DGArt_{\K}^{\le 0}$, the functor $f_{\ast}\colon \CDGA_A^{\le 0} \to \CDGA_B^{\le 0}$ satisfies the lifting problems introduced above.
The main idea to prove the claim relies on a technical lifting problem involving trivial idempotents, see Subsection~\ref{section.trivialidempotents}. By Lemma~\ref{prop.fixedloci} this is equivalent to solve the \emph{weak retractions of cofibrations} lifting problem.
The \emph{(CW,F)-factorization} and the \emph{(C,FW)-factorization} lifting problems in $\CDGA_{\K}^{\le 0}$ are solved in Theorem~\ref{thm.liftingfactorization2} and Theorem~\ref{thm.liftingfactorization} respectively. As a consequence, the lifting problem of (trivial) cofibrations is solved in Corollary~\ref{corollary.CW}.

All the lifting problems described above essentially deal with axioms of model categories, except for the one on retractions where some additional hypothesis have been assumed. Example~\ref{example.liftidempotents} will show that if we drop the assumption on the horizontal arrows, then the \emph{weak retractions of cofibrations} lifting problem may not admit solution even in the strong left-proper model category $\CDGA_{\K}^{\le 0}$.

\medskip
\subsection{Lifting of liftings}

\begin{lemma}\label{lemma.slashboxes}
Let $A\to B$ be a surjective morphism in $\DGArt_{\K}^{\le 0}$ and consider a fibration (respectively: trivial fibration) $p\colon S\to R$ in $\CDGA_A^{\le 0}$. Then the natural morphism
\[ S\to R\times_{R\otimes_AB}(S\otimes_AB) \]
is a fibration (respectively: trivial fibration).
\begin{proof}
Denote by $J$ the kernel of $A\to B$ and fix $i\le 0$. If $S^i\to R^i$ is surjective the following commutative diagram
\[ \xymatrix{	S^i\otimes_AJ \ar@{->}[r] \ar@{->}[d] & S^i \ar@{->}[r] \ar@{->}[d] & S^i\otimes_AB\ar@{->}[r] \ar@{->}[d] & 0 \ar@{->}[d] \\
R^i\otimes_AJ \ar@{->}[r] \ar@{->}[d] & R^i \ar@{->}[r] & R^i\otimes_AB \ar@{->}[r] & 0 \\
0 & & &	} \]
has exact rows and columns since the (graded) tensor product is right exact. By diagram chasing, it immediately follows the surjectivity of
\[ S^i\to R^i\times_{R^i\otimes_AB}(S^i\otimes_AB). \]
If moreover $p$ is a weak equivalence, then 
\[ R\times_{R\otimes_AB}(S\otimes_AB) \to R \]
is so, since trivial fibrations are stable under pullbacks. The statement follows by the \emph{2 out of 3} axiom.
\end{proof} 
\end{lemma}

\begin{theorem}\label{thm.slashboxes}
Let $f\colon A\to B$ be a surjective morphism in $\DGArt_{\K}^{\le 0}$. Consider a commutative diagram of solid arrows
\[ \xymatrix{	P_A \ar@{->}[dr]_{g_A} \ar@{->}[dd] \ar@{->}[rr] & & S_A \ar@{-}[d] \ar@{->}[dr]^{p_A} & \\
 & Q_A \ar@{->}[dd] \ar@{->}[rr] \ar@{-->}[ur]|-{h_A} & \ar@{->}[d] & R_A \ar@{->}[dd] \\
P_B \ar@{->}[dr]_{g_B} \ar@{-}[r] & \ar@{->}[r] & S_B \ar@{->}[dr]^{p_B} & \\
 & Q_B \ar@{->}[rr] \ar@{->}[ur]|-{h_B} & & R_B	} \]
in $\CDGA_A^{\le 0}$, where the upper square reduces to the bottom square applying $f_{\ast}$, and moreover the map $g_A$ is a cofibration (respectively: trivial cofibration) and the map $p_A$ is a trivial fibration (respectively: fibration). Then there exists a (dashed) lifting $h_A\colon Q_A\to S_A$ which reduces to $h_B$.
\begin{proof}
Consider the commutative diagram
\[ \xymatrix{	Q_A \ar@{-->}[r]_-{\varphi} \ar@{->}[d] \ar@/^1pc/@{->}[rr] & S_B \times_{R_B} R_A \ar@{->}[d] \ar@{->}[r] & R_A \ar@{->}[d] \\
Q_B \ar@{->}[r]^{h_B} \ar@/_1pc/@{->}[rr] & S_B \ar@{->}[r]^{p_B} & R_B	} \]
in $\CDGA_A^{\le 0}$, where the dashed morphism $\varphi\colon Q_A\to S_B \times_{R_B} R_A$ is given by the universal property of the pullback, which also ensures the existence of a (unique) map $S_A\to S_B\times_{R_B}R_A$ commuting with both $p_A$ and the projection $S_A\to S_B$.
By Lemma~\ref{lemma.slashboxes}, the commutative square of solid arrows
\[ \xymatrix{	P_A \ar@{->}[r] \ar@{->}[d] & S_A \ar@{->}[d] \\
Q_A \ar@{->}[r]_-{\varphi} \ar@{-->}[ur]^{h_A} & S_B\times_{R_B}R_A	} \]
admits the dashed lifting $h_A\colon Q_A\to S_A$, whence the statement.
\end{proof}
\end{theorem}

\begin{remark}
Notice that the statement of Theorem~\ref{thm.slashboxes} do not require any flatness hypothesis. This is due to the fact that we already assumed the existence of a fixed map $h_B\colon Q_B\to S_B$. On the other hand, if $P_A, S_A, Q_A, R_A$ are flat objects in $\CDGA_A^{\le 0}$, then by Lemma~\ref{lem.preservaweak} the functor $f_{\ast}$ preserves weak equivalences between them. Therefore the statement of Theorem~\ref{thm.slashboxes} implies that for any dashed lifting $h_B\colon Q_B\to S_B$ in the square
\[ \xymatrix{	P_B \ar@{->}[r] \ar@{->}[d] & S_B \ar@{->}[d] \\
Q_B \ar@{->}[r] \ar@{-->}[ur]^{h_B} & R_B		} \]
given by model category axioms, there exists a lifting $h_A\colon Q_A\to S_A$
\[ \xymatrix{	P_A \ar@{->}[r] \ar@{->}[d] & S_A \ar@{->}[d] \\
Q_A \ar@{->}[r] \ar@{-->}[ur]^{h_A} & R_A		} \]
which reduces to $h_B$ via $f_{\ast}$.
\end{remark}

\bigskip

\subsection{Lifting of trivial idempotents}\label{section.trivialidempotents}

The aim of this section can be explained as follows. Consider a map $A\to B$ in $\CDGA_{\K}^{\le 0}$ together with a commutative diagram of solid arrows
\[ \xymatrix{	P_A \ar@{->}[dr]_{e_A} \ar@{->}[dd] \ar@{->}[rr]^{g_A} & & R_A \ar@{-}[d] \ar@{-->}[dr]^{f_A} & \\
 & P_A \ar@{->}[dd] \ar@{->}[rr]^<<<<<<{g_A} & \ar@{->}[d] & R_A \ar@{->}[dd] \\
P_B \ar@{->}[dr]_{e_B} \ar@{-}[r]^{g_B} & \ar@{->}[r] & R_B \ar@{->}[dr]^{f_B} & \\
 & P_B \ar@{->}[rr]_{g_B} & & R_B	} \]
in $\CDGA_A^{\le 0}$, where $g_A$ is a cofibration, $e_A$ and $f_B$ are trivial idempotents and the arrows in the lower square are obtained applying the functor $-\otimes_AB$ to the ones of the upper square. The goal of this section is to prove the existence of a trivial idempotent $f_A\colon R_A\to R_A$ fitting the diagram above. In other terms, we are looking for a trivial idempotent $f_A$ whose reduction is $f_B$, and such that $f_Ag_A=g_Ae_A$, see Theorem~\ref{thm.liftingidempotents}.

The following example shows that if we do not assume the idempotent $f_B$ to be a weak equivalence, then the lifting problem above may not admit a solution.

\begin{example}\label{example.liftidempotents}
Let $A=\faktor{\K[\varepsilon]}{(\varepsilon^2)}$, let $B=\K$, and consider $R_B=\K[x,y]\in\CDGA_{\K}^{\le 0}$ where $\deg(x)=0$, $\deg(y)=1$, and $dy=0$. Then define $R_A=R_B\otimes_{\K}A$ as a commutative graded algebra, endowed with the differential $d_Ay=\varepsilon x$. Clearly $R_A\otimes_AB = R_B$. Moreover, consider the (non-trivial) idempotent $f_B\colon R_B\to R_B$ defined by $f_B(x)=x$, $f_B(y)=0$; let $P_A=R_A$ and assume the maps $g_A\colon R_A\to R_A$ and $e_A\colon R_A\to R_A$ to be the identity morphism. By contradiction, assume the existence of an idempotent $f_A\colon R_A\to R_A$ lifting $f_B$; then $f_A$ has to be defined by
\[ f_A\colon \begin{cases}
x \mapsto x + \varepsilon z \\
y \mapsto \varepsilon wy
\end{cases} \]
for some $w,z\in A$. Now notice that the relations
\[ f_A(d_Ay) = f_A(\varepsilon x) = \varepsilon x \qquad \mbox{ and } \qquad d_A f_A(y) = d_A(\varepsilon wy) = 0 \]
imply that such $f_A$ is not a morphism in $\CDGA_A^{\le 0}$ independently of the choice of $w,z$.
\end{example}

The result explained above requires several preliminary results.
Recall that $\CGA^{\le 0}_{\K}$ denotes the category of commutative graded algebras over $\K$ concentrated in non-positive degrees.

\begin{lemma}\label{lemma.liftingidempotents}
Given $A\in\CDGA_{\K}^{\le 0}$, consider a commutative diagram of solid arrows
\[ \xymatrix{	P\ar@{->}[r]^{g} \ar@{->}[d]_{i} & E\ar@{->}[d]^{p} \\
C\ar@{->}[r]_{f} \ar@{-->}[ur] & D	} \]
in $\CDGA_A^{\le 0}$. If $i$ is a cofibration and $p$ is surjective, then there exists the dotted lifting $\gamma\colon C\to E$ in the category $\CGA_{\K}^{\le 0}$.
\begin{proof}
Consider the \textbf{killer algebra} $A[d^{-1}]\in\CDGA_{A}^{\le 0}$. Recall that the natural inclusion $\alpha\colon A\to A[d^{-1}]$ is a morphism of DG-algebras and the natural projection $\beta\colon A[d^{-1}]\to A$ is a morphism of graded algebras; moreover $\beta\alpha$ is the identity on $A$.
Now, the morphism
\[ E\otimes_AA[d^{-1}]\xrightarrow{p\otimes\id} D\otimes_A A[d^{-1}] \]
is a trivial fibration and then there exists a commutative diagram
\[ \xymatrix{	P\ar@{->}[rr]^{\alpha g} \ar@{->}[d]_{i} & & E\otimes_AA[d^{-1}] \ar@{->}[d]^{p\otimes \id} \\
C\ar@{->}[rr]_{\alpha f} \ar@{->}[urr]^{\varphi} & & D\otimes_AA[d^{-1}]	} \]
in $\CDGA_A^{\le 0}$. It is now sufficient to take $\gamma = \beta\varphi$.
\end{proof}
\end{lemma}

\begin{proposition}[Algebraic lifting of idempotents]\label{prop.liftingidempotents}
Let $i\colon A\to P$ be a morphism in $\CGA^{\le 0}_{\K}$, and $J\subseteq A$ a graded ideal satisfying $J^2=0$. Moreover, consider a morphism $g\colon P\to P$ together with an idempotent $e\colon A\to A$ in $\CGA^{\le 0}_{\K}$ such that $e(J)\subseteq J$ and $gi=ie$. Denote by
\[ \bar{g}\colon P/i(J)P\to P/i(J)P \]
the factorization to the quotient, and assume that $\bar{g}^2=\bar{g}$. 
Then there exists a morphism $f\colon P\to P$ in $\CGA^{\le 0}_{\K}$ such that $f^2=f$, $fi=ie$, and $\bar{f}=\bar{g}$, i.e. $f\equiv g$ $\pmod{i(J)P}$.
\begin{proof}
First notice that the condition $gi=ie$ implies that $g(i(J)P)\subseteq i(e(J))g(P)\subseteq i(J)P$, so that the induced morphism $\bar{g}$ is well defined. For notational convenience, in the rest of the proof we shall write $JP$ in place of $i(J)P$, since no confusion occurs. Notice that for every $x\in JP$ we have $g^2(x)=g(x)$; in fact take $x=i(a)p$, with $a\in J$ and $p\in P$, then 
\[ g^2(i(a)p)-g(i(a)p) = i(e^2(a))g^2(p)-i(e(a))g(p) = i(e(a))(g^2(p)-g(p))\in J^2P=0 \]
since by assumption $g^2(p)-g(p)\in JP$.
Now denote by $\phi=g^2-g\colon P\to P$. By hypothesis we have
\[ \phi i = g^2i - gi = gie-ie = ie^2-ie = 0 \, ,\]
\[ \phi(P)\subseteq JP,  \mbox{ and } g\phi=\phi g \, .\]
Notice that the morphism $\phi$ is a $g$-derivation of degree $0$; in fact for every $p, q\in P$
\[ \phi(pq)=g^2(p)g^2(q)-g(p)g(q)=g^2(p)\phi(q)+\phi(p)g(q)=g(p)\phi(q)+\phi(p)g(q),\]
where the last equality follows since $g^2(p)\phi(q) = g(p)\phi(q)$, being $\phi(p)\phi(q)\in J^2P=0$.
Now, define $\psi\colon P\to JP$ as $\psi = \phi - g\phi - \phi g = - g + 3g^2 - 2g^3$, and notice that 
\begin{enumerate}
\item $\psi(J)=0$, $\psi i=0$ because $\phi i=0$,
\item $\psi^2=0$ and $g^2\psi=g\psi=\psi g=\psi g^2$ because $\phi\psi=\psi\phi=0$,
\item $\psi$ is a $g$-derivation,
\item $\psi - g\psi - \psi g = \phi - 4g\phi + 4g^2\phi = \phi + 4\phi^2 = \phi$.
\end{enumerate}
In particular,
\[ (g+\psi)^2-(g+\psi)=g^2+g\psi+\psi g+\psi^2-g-\psi=\phi+g\psi+\psi g-\psi=0 \]
and
\[ (g+\psi)i = 3g^2i-2g^3i = 3ie^2-2ie^3 = 3ie-2ie = ie \, .\] 
Therefore, to obtain the statement it is sufficient to define $f=g+\psi=3g^2-2g^3$, which is a morphism in $\CGA^{\le 0}_{\K}$ satisfying the required properties.
\end{proof}
\end{proposition}

\begin{remark}
The previous result actually holds even if we replace $\CGA^{\le 0}_{\K}$ with the category of unitary graded commutative rings.
\end{remark}

For every morphism $A\to B$ in $\DGCA_{\K}^{\le 0}$ and every $M\in \DGMod(B)$ we shall denote by 
$\Der^*_A(B,M)$ the differential graded $B$-module of $A$-derivations $B\to M$. 
For every pair of morphisms $A\to B\xrightarrow{f} C$ of commutative differential graded algebras we shall denote by $\Der^*_A(B,C;f)$ the module of derivations, where the $B$-module structure on $C$ is induced by the morphism $f$

\begin{remark}\label{rem.propertyofdef}
In the sequel we shall use in force the following basic facts:
\begin{enumerate}

\item for every cofibrant $A$-algebra $B\in \DGCA_{A}^{\le 0}$ and every surjective quasi-isomorphism $M\to N$ in 
$\DGMod(B)$ the induced map $\Der^*_A(B,M)\to \Der^*_A(B,N)$ is a surjective quasi-isomorphism;

\item for every weak equivalence $B\to C$ of cofibrant objects in $\DGCA_{A}^{\le 0}$ and every $M\in \DGMod(C)$ the induced map $\Der^*_A(C,M)\to \Der^*_A(B,M)$ is a weak equivalence.
\end{enumerate} 

The above properties are well known \cite[Sec.7]{Hin} and in any case easy to prove as the consequence of the following straightforward facts:

\begin{itemize} 

\item A morphism in $\DGCA_{\K}^{\le 0}$ is a weak equivalence (resp.: cofibration, trivial fibration) if and only if it is a weak equivalence (resp.: cofibration, trivial fibration) as a morphism in  $\DGCA_{\K}$;

\item for every integer $n$ there exists a natural bijection between $Z^n(\Der^*_A(B,M))$ and the set of liftings in the  
obvious commutative solid diagram 
\[ \xymatrix{A\ar[d]\ar[r]&B\oplus M[n]\ar[d]\\
B\ar@{-->}[ur]\ar[r]_{\id_B}&B}\]
in $\DGCA_{\K}$;
\item for every integer $n$ there exists a natural bijection between $\Der^n_A(B,M)$ and  the set of liftings in the  
obvious commutative solid diagram 
\[ \xymatrix{A\ar[d]\ar[r]&B\oplus \cone(\id_{M[n-1]})\ar[d]\\
B\ar@{-->}[ur]\ar[r]_{\id_B}&B}\;;\]
in $\DGCA_{\K}$, and the differential of  $\Der^*_A(B,M)$ is induced (up to sign) by the natural morphisms of $B$-modules
\[ \cone(\id_{M[n-1]})\to M[n]\to \cone(\id_{M[n]})\,.\]

\end{itemize}
\end{remark}

\begin{lemma}\label{lem.Dacyclic}
Consider a morphism of retractions in $\CDGA_{\K}^{\le 0}$
\[ \xymatrix{	Q\ar@{->}[r]^{j}\ar@{->}[d] & P\ar@{->}[r]^{q}\ar@{->}[d] & Q\ar@{->}[d] \\
S\ar@{->}[r]^{i} & R\ar@{->}[r]^{p} & S		} \]
and define $f = ip\colon R\to R$ and $e=jq\colon P\to P$. Let $\alpha\in\Der^{\ast}_P(R, R; f)$ and $\beta\in\Der^{\ast}_Q(S,S)$ be derivations such that the diagram
\[ \xymatrix{	R\ar@{->}[r]^{p} \ar@{->}[d]^{\alpha} & S\ar@{->}[r]^{i} \ar@{->}[d]^{\beta} & R\ar@{->}[d]^{\alpha} \\
R\ar@{->}[r]^{p} & S\ar@{->}[r]^{i} & R	} \]
commutes. Then $i\beta p\in\Der^{\ast}_P(R, R; f)$ and, setting $\gamma =\alpha-2i\beta p$ we have
\[ \gamma - \gamma f - f\gamma = \alpha. \]
Conversely, given any $\gamma\in\Der^{\ast}_P(R, R; f)$, the $P$-linear $f$-derivation $\alpha = \gamma - \gamma f - f\gamma$ satisfies
\[ \alpha(\ker(p)) \subseteq \ker(p), \qquad \alpha(i(S)) \subseteq i(S) \]
and factors through a derivation $\beta\colon S\to S$ as above.
\end{lemma}

\begin{proof}
Observe that $i\beta p$ is an $f$-derivation being $f=ip$. Moreover, since $pi=\id$ we have
\[ \begin{aligned}
\gamma - \gamma f - f\gamma &= \alpha - 2i\beta p - \alpha ip + 2i\beta pip - ip\alpha + 2ipi\beta p =\\
&= \alpha - 2i\beta p + 2i\beta p + 2i\beta p - 2\alpha ip = \alpha.
\end{aligned} \]
Conversely, take $\gamma\in\Der^{\ast}_P(R, R; f)$ and define $\alpha = \gamma - \gamma f - f\gamma$. Now, observe that $\ker(p)=\ker(f)$, and since
\[ f\alpha(x) = f\gamma(x) - f^2\gamma(x) - \gamma f(x) = \gamma f(x) \]
we have $\alpha(\ker(p)) \subseteq \ker(p)$. Similarly, since $i(S) = f(R)$ the chain of equalities
\[ \alpha f = \gamma f - \gamma f^2 - f\gamma f = -f\gamma f \]
implies that $\alpha(i(S)) \subseteq i(S)$. Notice that $\beta = p\alpha i = -p\gamma i$, so that $\alpha f = i\beta p$. To conclude the proof recall that the restriction of $f$ to $S$ is the identity, therefore $\beta$ is a $P$-linear derivation.
\end{proof}

\begin{proposition}\label{prop.Dacyclic}
Let $e\colon P\to P$ and $f\colon R\to R$ be trivial idempotents in $\CDGA_{\K}^{\le 0}$, and consider a cofibration $g\colon P\to R$ in $\CDGA_{\K}^{\le 0}$ such that $ge=fg$. Then the subcomplex
\[ D = \left\{ \gamma\in\Der^{\ast}_P(R,R;f) \,\,\vert\,\, \gamma = f\gamma + \gamma f \right\} \subseteq \Der^{\ast}_P(R,R;f) \]
is acyclic.
\begin{proof}
We can write $f = ip$ and $e=jq$ for a morphism between retractions in $\CDGA_{\K}^{\le 0}$
\[ \xymatrix{	Q\ar@{->}[r]^{j}\ar@{->}[d]_{\bar{g}} & P\ar@{->}[r]^{q}\ar@{->}[d]^{g} & Q\ar@{->}[d]^{\bar{g}} \\
S\ar@{->}[r]^{i} & R\ar@{->}[r]^{p} & S		} \]
where both $g$ and $\bar{g}$ are cofibrant objects.
The pushout of $\bar{g}$ along $j$ gives an extension of the diagram above to
\[ \xymatrix{	Q\ar@{->}[r]^{j}\ar@{->}[d]_{\bar{g}} & P\ar@{->}[d]^{\tilde{g}}\ar@{->}[rr]^{q}\ar@{->}[dr]^{g} & & Q\ar@{->}[d]^{\bar{g}} \\
S\ar@/_2pc/@{->}[rr]_{i} \ar@{->}[r]^{\tilde{i}} & S\otimes_QP\ar@{->}[r]^{\tau} & R\ar@{->}[r]^{p} & S		} \]
in $\CDGA_{\K}^{\le 0}$. Since $i$ and $p$ are retracts of $f$, they are weak equivalences; in particular $p$ is a trivial fibration. The same holds for $j$ and $q$, so that $\tilde{i}$ is a weak equivalence. It then follows that $\tau$ is a weak equivalence between cofibrant objects in $\CDGA_P^{\le 0}$. By Lemma~\ref{lem.Dacyclic} there exists a short exact sequence
\[ 0\to D\to \Der^{\ast}_P(R,R;f)\xrightarrow{\gamma\mapsto (\gamma f+f\gamma-\gamma,p\gamma i)} K\to 0\]
of DG-modules over $R$, where
\[ K = \left\{ (\alpha,\beta)\in\Der^{\ast}_P(R,R;f)\times\Der^{\ast}_Q(S,S) \,\,\vert\,\, \beta p = p\alpha, i\beta = \alpha i \right\}. \]
Since $p$ is a trivial fibration and $R$ is cofibrant, the map
\[ \begin{aligned}
p_{\ast}\colon \Der^{\ast}_P(R, R; f) &\to \Der^{\ast}_P(R, S; pf) \\
\gamma &\mapsto p\gamma
\end{aligned} \]
is a trivial fibration by Remark~\ref{rem.propertyofdef}; here we should think of $S$ as an object in $\CDGA_P^{\le 0}$ via the map $\bar{g}q\colon P\to S$. Now recall that $pf=p$, and since $\tau$ is a weak equivalence between cofibrant objects in $\CDGA_P^{\le 0}$, then the map
\[ \begin{aligned}
\tau^{\ast}\colon \Der^{\ast}_P(R, S; pf)=\Der_P^{\ast}(R, S; p) &\to \Der^{\ast}_P(S\otimes_QP, S; p\tau) =\Der_Q^{\ast}(S, S; \id) \\
\gamma &\mapsto \gamma \tau
\end{aligned} \]
is a weak equivalence.
Therefore, in order to prove the statement it is sufficient to prove that also the projection $K\to \Der^{\ast}_Q(S,S)$ is a weak equivalence. Since every $\beta\in\Der^{\ast}_Q(S,S)$ lifts to $(i\beta p, \beta)\in K$, we have a short exact sequence
\[ 0\to H\to K\to \Der^{\ast}_Q(S,S)\to 0, \]
where
\[ H = \left\{ \alpha\in\Der^{\ast}_P(R,R; f) \,\,\vert\,\, \alpha i = p\alpha = 0 \right\} = \left\{ \alpha\in\Der^{\ast}_P(R, \ker\{p\}) \,\,\vert\,\, \alpha i = 0\right\}, \]
where the $R$-module structure on $\ker\{p\}$ is induced via the morphism $f$. Therefore we have a short exact sequence
\[ 0\to H\to \Der^{\ast}_P(R,\ker\{p\}) \xrightarrow{\tilde{i}^{\ast}}  \Der^{\ast}_P(S\otimes_QP,\ker\{p\}) = \Der^{\ast}_Q(S,\ker\{p\}) \to 0 \]
and  the map $\tilde{i}^{\ast}$ is a trivial fibration. 
It follows that $H$ is an acyclic complex, so that the projection $K\to \Der^{\ast}_Q(S,S)$ is a weak equivalence as required.
\end{proof}
\end{proposition}

\begin{theorem}[\textbf{Lifting of trivial idempotents}]\label{thm.liftingidempotents}
Let $A\to B$ be a surjective morphism in $\DGArt_{\K}^{\le 0}$. Moreover, consider a cofibration $g_A\colon P_A\to R_A$ between flat objects in $\CDGA_A^{\le 0}$, together with a trivial idempotent $e_A\colon P_A\to P_A$; denote by
\[ g_B\colon P_B = P_A\otimes_A B\to R_A\otimes_AB = R_B \qquad e_B\colon P_B\to P_B \]
the pushout cofibration and the pushout idempotent in $\CDGA_B^{\le 0}$. Moreover, let $f_B\colon R_B\to R_B$ be a trivial idempotent in $\CDGA_B^{\le 0}$ satisfying $f_Bg_B = g_Be_B$.
Then there exists a trivial idempotent $f_A\colon R_A\to R_A$ in $\CDGA_{A}^{\le 0}$ lifting $f_B$ such that $f_Ag_A = g_Ae_A$.
\begin{proof}
We proceed by induction on the length of $A$. First notice that it is not restrictive to assume the morphism $A\to B$ comes from a small extension
\[ 0 \to \K t \to A \to B \to 0 \]
in $\DGArt_{\K}^{\le 0}$, for some cocycle $t$ in the maximal non-zero power of the maximal ideal $\mathfrak{m}_A$.
Notice that  $\K t$ is a complex concentrated in degree $i=\deg(t)$, and $\K t \to A$ is the inclusion. In fact, every surjective map in $\DGArt_{\K}^{\le 0}$ factors in a sequence of small extensions as above.

Since $g_A$ is a cofibration, the diagram of solid arrows
\[ \xymatrix{	P_A \ar@{->}[r]^{e_A} \ar@{->}[d]^{g_A} & P_A \ar@{->}[r]^{g_A} & R_A \ar@{->}[d] \\
R_A \ar@{->}[r] \ar@{-->}[urr]|-{r} & R_B \ar@{->}[r]^{f_B} & R_B 	} \] 
admits the dotted lifting in $\CGA_{\K}^{\le 0}$ by Lemma~\ref{lemma.liftingidempotents}. This means that $f_B$ lifts to a morphism of graded algebras $r\colon R_A\to R_A$ satisfying $rg_A=g_Ae_A$. Moreover, by Proposition~\ref{prop.liftingidempotents} we may assume $r^2 = r$. Now set $P=P_A\otimes_A\K$ and $R=R_A\otimes_A\K$; denote by $d\in\Hom_A^{1}(R_A,R_A)$ the differential of $R_A$. Then
\[ dr - rd = \iota\psi \pi, \mbox{ for some } \psi\in\Der_P^{1}(R, R; f) \]
where $\iota\colon R[-i]\cdot t\to R_A$ is the morphism induced by the small extension while $R_A\xrightarrow{\pi} R$ is the natural projection. It follows that $\psi$ is a cocycle in the complex $D$ of Proposition~\ref{prop.Dacyclic}. In fact, setting $f=f_B\otimes_B\K$, we have $\iota f = r\iota$ and $\pi r = f\pi$ by construction, so that
\[ \iota(d\psi + \psi d)\pi = d(dr - rd) + (dr - rd)d = 0, \]
\[ \iota(f\psi + \psi f)\pi = rdr - r^2d + dr^2 - rdr = dr - rd = \iota\psi \pi. \]
Therefore there exists $h\in\Der^0_P(R,R; f)$ such that
\[dh - hd = \psi, \qquad \qquad fh + hf - h = 0. \]
Setting $f_A = r - \iota h\pi$ we have that $f_A$ is a morphism of graded algebras. Moreover
\[ f_A^2 - f_A = \iota(-fh - hf + h)\pi = 0 \, , \qquad \qquad df_A - f_Ad = \iota(\psi - dh + hd)\pi = 0 \, ,\]
and the image of $\pi g_A$ is contained in $P$, so that $ih\pi g_A=0$ being $h$ a $P$-linear derivation. It follows that $f_A$ is an idempotent in $\CDGA_A^{\le 0}$ satisfying $f_Ag_A=g_Ae_A$. By Corollary~\ref{cor.naka} the morphism $f_A$ is a weak equivalence and the statement follows.
\end{proof}
\end{theorem}

\bigskip

\subsection{Lifting of factorizations}\label{section.factorization}

The main goal of this section is to show that for every $A\in\DGArt_{\K}^{\le 0}$, for every flat object $P\in\CDGA_A^{\le 0}$ and for every trivial cofibration $\bar{f}\colon P\otimes_A\K \to \bar{Q}$ in $\CDGA_{\K}^{\le 0}$, there exists a trivial cofibration $f\colon P\to Q$ in $\CDGA_A^{\le 0}$ lifting $\bar{f}$. Actually we shall prove stronger results (see Theorem~\ref{thm.liftingfactorization} and Theorem~\ref{thm.liftingfactorization2}), and the required statement will follow, see Corollary~\ref{corollary.CW}.

\begin{theorem}\label{thm.liftingfactorization} 
Let $A\to B$ be a surjection in $\DGArt_{\K}^{\le 0}$ and consider a morphism $f\colon P\to M$ in $\CDGA_A^{\le 0}$ between flat objects.
Then every (C,FW)-factorization of the reduction
\[ \bar{f}=f\otimes_AB\colon \bar{P}=P\otimes_AB\to \bar{M}=M\otimes_AB \]
lifts to a factorization of $f$; i.e. for every factorization $\bar{P}\xrightarrow{\sC} \bar{Q}\xrightarrow{\sF\sW} \bar{M}$ of $\bar{f}$ there exists a commutative diagram
\[ \xymatrix{  	P \ar@{->}[d]\ar@{->}[r]^{\sC} & Q \ar@{->}[d]\ar@{->}[r]^{\sF\sW} & M \ar@{->}[d] \\ 
\bar{P} \ar@{->}[r]^{\sC} & \bar{Q}\ar@{->}[r]^{\sF\sW} & \bar{M} 	} \]
in $\CDGA_A^{\le 0}$, where the upper row reduces to the bottom row applying the functor $-\otimes_AB$ and the vertical morphisms are the natural projections.
\begin{proof}
We have a commutative diagram 
\[ \xymatrix{	P\ar@{->}[r]^{g} \ar@{->}[d] \ar@/^2pc/[rr]^{f} & \bar{Q}\times_{\bar{M}}M \ar@{->}[d] \ar@{->}[r]^{\sF\sW} & M \ar@{->}[d] \\
\bar{P} \ar@{->}[r]^{\sC} & \bar{Q} \ar@{->}[r]^{\sF\sW} & \bar{M}	} \] 
in $\CDGA_A^{\le 0}$. Taking a factorization  of $g$ we get 
\[ \xymatrix{	 & D \ar@{->}[d]^{\sF\sW} \ar@/^1pc/[dr]^{\sF\sW} & \\
P\ar@{->}[r]^{g} \ar@{->}[d] \ar@/^1pc/[ur]^{\sC} & \bar{Q}\times_{\bar{M}}M \ar@{->}[d] \ar@{->}[r]^{\sF\sW} & M \ar@{->}[d] \\
\bar{P} \ar@{->}[r]^{\sC} & \bar{Q} \ar@{->}[r]^{\sF\sW} & \bar{M}	} \]
Notice that the composite map $D\to \bar{Q}$ is surjective. Now $D$ and $M$ are $A$-flat and therefore the morphism $\bar{D}=D\otimes_A\K\to \bar{M}$ is a weak equivalence,
and since it factors through $\bar{D}\to  \bar{Q}\xrightarrow{\sF\sW}\bar{M}$, the surjective map $p\colon \bar{D}\to \bar{Q}$ is a trivial fibration.
It follows the existence of a section $s\colon \bar{Q}\to \bar{D}$
commuting with the maps $\bar{P}\to \bar{D}$ and $\bar{P}\to \bar{Q}$.
Since $P\to D$ is a cofibration, by Theorem~\ref{thm.liftingidempotents} the idempotent $\bar{e}=sp\colon \bar{D}\to \bar{D}$ lifts to 
an idempotent of $e\colon D\to D$. Setting $Q=\{x\in D\mid e(x)=x\}$, by Proposition~\ref{prop.fixedloci} we have that $Q\otimes_A\K=\bar{Q}$ and $P\to Q$ is a cofibration because it is a retract of $P\to D$. 
\end{proof}
\end{theorem}

\begin{corollary}\label{cor.flattrivialcofib}
Let $A\in \DGArt_{\K}^{\le 0}$ and consider a morphism $f\colon P\to M$ in $\CDGA_A^{\le 0}$ between flat objects.
Then $f$ is a cofibration if and only if its reduction $\bar{f}\colon P\otimes_A\K \to M\otimes_A \K$ is a cofibration in $\CDGA_{\K}^{\le 0}$.
\begin{proof}
First assume that $\bar{f}$ is a cofibration; by Theorem \ref{thm.liftingfactorization} there exists a commutative diagram
\[ \xymatrix{	P \ar@{->}[r]|-{\sC} \ar@/^2pc/@{->}[rrr]^{f} \ar@{->}[d] & Q \ar@{->}[rr]|-{\sF\sW} \ar@{->}[d] & & M \ar@{->}[d] \\
P\otimes_A\K \ar@{->}[r]^{\bar{f}} & M\otimes_A\K \ar@{->}[rr]^{\id} & & M\otimes_A\K 	} \]
in $\CDGA_A^{\le 0}$, where the upper row reduces to the bottom row via the functor $-\otimes_A\K$. Moreover, by flatness, Corollary~\ref{cor.naka} implies that the trivial fibration $Q\to M$ is in fact an isomorphism, so that $f$ is obtained as a cofibration followed by an isomorphism, whence the thesis. The converse holds since the class of cofibrations is closed under pushouts.
\end{proof} 
\end{corollary}

\begin{theorem}\label{thm.liftingfactorization2}
Let $A\to B$ be a surjection in $\DGArt_{\K}^{\le 0}$ and consider a morphism $f\colon P\to M$ in $\CDGA_A^{\le 0}$ between flat objects.
Then every (CW,F)-factorization of the reduction
\[ \bar{f}=f\otimes_AB\colon \bar{P}=P\otimes_AB\to \bar{M}=M\otimes_AB \]
lifts to a factorization of $f$; i.e. for every factorization $\bar{P}\xrightarrow{\sC\sW} \bar{Q}\xrightarrow{\sF} \bar{M}$ of $\bar{f}$ there exists a commutative diagram
\[ \xymatrix{  	P \ar@{->}[d]\ar@{->}[r]^{\sC\sW} & Q \ar@{->}[d]\ar@{->}[r]^{\sF} & M \ar@{->}[d] \\ 
\bar{P} \ar@{->}[r]^{\sC\sW} & \bar{Q}\ar@{->}[r]^{\sF} & \bar{M} 	} \]
in $\CDGA_A^{\le 0}$, where the upper row reduces to the bottom row applying the functor $-\otimes_AB$ and the vertical morphisms are the natural projections.
\begin{proof}
The proof is essentially the same as in Theorem~\ref{thm.liftingfactorization}.
We have a commutative diagram 
\[ \xymatrix{	P\ar@{->}[r]^{g} \ar@{->}[d]^{\sF} \ar@/^2pc/[rr]^{f} & \bar{Q}\times_{\bar{M}}M \ar@{->}[d]^{\sF} \ar@{->}[r]^{\sF} & M \ar@{->}[d]^{\sF} \\
\bar{P} \ar@{->}[r]^{\sC\sW} & \bar{Q} \ar@{->}[r]^{\sF} & \bar{M}	} \] 
in $\CDGA_A^{\le 0}$. Taking a factorization  of $g$ we get 
\[ \xymatrix{	 & D \ar@{->}[d]^{\sF} \ar@/^1pc/[dr]^{\sF} & \\
P\ar@{->}[r]^{g} \ar@{->}[d]^{\sF} \ar@/^1pc/[ur]^{\sC\sW} & \bar{Q}\times_{\bar{M}}M \ar@{->}[d]^{\sF} \ar@{->}[r]^{\sF} & M \ar@{->}[d]^{\sF} \\
\bar{P} \ar@{->}[r]^{\sC\sW} & \bar{Q} \ar@{->}[r]^{\sF} & \bar{M}	} \]
Notice that the composite map $D\to \bar{Q}$ is surjective in negative degrees and hence a fibration. Moreover, the morphism $\bar{P}\to \bar{D}=D\otimes_A\K$ is a trivial cofibration since $P\to D$ is so. Now since $\bar{P}\to \bar{Q}$ factors through $\bar{P}\to \bar{D}$, the map 
$p\colon \bar{D}\to \bar{Q}$ is a trivial fibration.
It follows the existence of a section $s\colon \bar{Q}\to \bar{D}$ commuting with the maps $\bar{P}\to \bar{D}$ and $\bar{P}\to \bar{Q}$.
Since $P\to D$ is a cofibration, by Theorem~\ref{thm.liftingidempotents} the idempotent $\bar{e}=sp\colon \bar{D}\to \bar{D}$ lifts to 
an idempotent of $e\colon D\to D$. Setting $Q=\{x\in D\mid e(x)=x\}$, by Proposition~\ref{prop.fixedloci} we have that $Q\otimes_A\K=\bar{Q}$ and $P\to Q$ is a cofibration because it is a retract of $P\to D$. 
\end{proof}
\end{theorem}

By Theorem~\ref{thm.liftingfactorization2} it follows the result that we claimed at the beginning of the section.

\begin{corollary}\label{corollary.CW}
Let $A\in \DGArt_{\K}^{\le 0}$ and consider a flat object $P\in\CDGA_A^{\le 0}$.
For every trivial cofibration $\bar{f}\colon \bar{P}=P\otimes_A\K\to \bar{Q}$ in $\CDGA_{\K}^{\le 0}$ there exist a flat object $Q\in\CDGA_A^{\le 0}$ such that $Q\otimes_A\K=\bar{Q}$ and a trivial cofibration $f\colon P\to Q$ lifting $\bar{f}$.
\begin{proof}
It is sufficient to apply Theorem~\ref{thm.liftingfactorization2} to the factorization $\bar{P} \xrightarrow{\sC\sW} \bar{Q} \xrightarrow{\sF} 0$.
\end{proof}
\end{corollary}

\begin{corollary}
Let $A\in \DGArt_{\K}^{\le 0}$ and consider a cofibrant object $Q\in\CDGA_A^{\le 0}$.
For every trivial cofibration $\bar{f}\colon \bar{P}\to \bar{Q}=Q\otimes_A\K$ in $\CDGA_{\K}^{\le 0}$ there exist a flat object $P\in\CDGA_A^{\le 0}$ such that $P\otimes_A\K=\bar{P}$ and a lifting of $\bar{f}$ to a trivial cofibration $f\colon P\to Q$.
\begin{proof}
Since $\bar{P}$ is fibrant the diagram of solid arrows
\[ \xymatrix{	\bar{P} \ar@{->}[d]_{\bar{f}} \ar@{->}[r]^{\id} & \bar{P} \ar@{->}[d] \\
\bar{Q} \ar@{->}[r] \ar@{-->}[ur]^{\bar{p}} & 0	} \]
admits the dotted lifting $\bar{p}\colon \bar{Q}\to \bar{P}$ in $\CDGA_{\K}^{\le 0}$. In particular, $\bar{P}$ is the fixed locus of the trivial idempotent $\bar{e} = \bar{f}\circ\bar{p} \colon \bar{Q}\to \bar{Q}$. By Theorem~\ref{thm.liftingidempotents} there exists a trivial idempotent $e\colon Q\to Q$ whose fixed locus $P = \{x\in Q\mid e(x)=x\}$ satisfies $P\otimes_A\K=\bar{P}$, see Proposition~\ref{prop.fixedloci}. The lifting of $\bar{f}$ is given by Theorem~\ref{thm.liftingfactorization2}.
\end{proof}
\end{corollary}

\bigskip

\section{Deformations of DG-algebras}

Following the general construction of Section~\ref{section.deformationmorphism} for every $R=(\K\to R)$ in $\CDGA_{\K}^{\le 0}$ we can 
consider the functor $\Def_R$ of its deformations in the strong left-proper model category $\CDGA_{\K}^{\le 0}$, defined in the 
category $\bM(\K)$. Recall that the above functor is homotopy invariant (Theorem~\ref{thm.homotopyequivalence}), i.e., for every 
weak equivalence $R\to S$ and every $A\in \bM(\K)$ the natural map $\Def_R(A)\to \Def_S(A)$ is bijective.  
In order to prove some additional interesting properties, in view of Corollary~\ref{cor.naka} and the results of Section~\ref{sec.lifting}, we consider the restricted functor\footnote{We shall see later that for every $A\in \DGArt_{\K}^{\le 0}$ the class $\Def_R(A)$ is not proper.}
\[ \Def_R\colon  \DGArt_{\K}^{\le 0}\to \Set\] 
of (set-valued) derived deformations of $R$. The main goal of this section is to prove that: 
\begin{enumerate}

\item if $R$ and $A$ are concentrated in degree $0$ then $\Def_R(A)$ is naturally isomorphic to the set of isomorphism classes of deformations defined in the classical sense: 
\[ \faktor{\Def_R(A)\cong \left\{		\begin{matrix}
\mbox{commutative flat $A$-algebras $R_A$ together with} \\
\mbox{an isomorphism $R_A\otimes_A\K\cong R$ of $\K$-algebras}
\end{matrix}	\right\} }{\text{isomorphism}}	\,;	\]

\item every deformation of a cofibrant DG-algebra may be obtained by a perturbation of the differential; 

\item  if $S\to R$ is a cofibrant
resolution, then the DG-Lie algebra of derivations of $S$ controls the functor $\Def_R$.
\end{enumerate}

It is interesting to point out that the above point (3) requires DG-algebras in non-positive degrees, and its analog fails in the category $\DGCA_{\K}$, see e.g. Example~\ref{ex.esempiononcofibrante}.

\medskip
\subsection{Strict deformations}\label{section.strictdeformations}

In this subsection we introduce the notion of strict deformations in $\CDGA_{\K}^{\le 0}$. It is a purely technical notion used in order to study deformations of algebras of special type: as we shall see in Example~\ref{ex.strictnogood} strict deformations are not homotopy invariant and then unsuitable to study deformations in full generality.

\begin{definition}\label{def.strictdeformationCDGA}
Given  $R\in\CDGA_{\K}^{\le 0}$, the class of strict deformations of $R$ over $A\in \DGArt_{\K}^{\le 0}$ is 
defined by
\[ \faktor{ \D_R(A) = \left\{		\begin{matrix}
\mbox{morphisms $R_A\to R$ in $\CDGA_A^{\le 0}$ such that $R_A$ is flat,} \\
\mbox{and the reduction $R_A\otimes_A\K\to R$ is an isomorphism}
\end{matrix}	\right\} }{\cong}\;.		\]
Two strict deformations $R_A\to R$ and $R_A'\to R$ are isomorphic if and only if there exists an isomorphism $R_A\xrightarrow{\cong} R_A'$ in $\CDGA_A^{\le 0}$ such that the diagram
\[ \xymatrix{ R_A \ar@{->}[rr]^{\cong} \ar@{->}[dr] & & R_A' \ar@{->}[dl] \\
 & R & 	} \]
commutes.
\end{definition}

It is plain that for every  $R\in\CDGA_{\K}^{\le 0}$ and every $A\in \DGArt_{\K}^{\le 0}$ there exists  a natural map
\[ \eta_A\colon \D_R(A) \longrightarrow \Def_R(A),\qquad (R_A\to X)\mapsto (R_A\to X)\,.\]
Whenever $A\in \Art_{\K}$, by Corollary~\ref{cor.degreewiseflat}, the restriction to the grade 0 component gives also a natural map $D_R(A)\to D_{R^0}(A)$. 

\begin{example}[Classical infinitesimal deformations as strict deformations]\label{example.classicaldef}
Consider an object $R$ in $\CDGA_{\K}^{\le 0}$ together with an Artin ring $A\in\Art_{\K}$, and assume that $R$ is concentrated in degree $0$. The same argument used in the proof of Corollary~\ref{cor.flatnessandrelation}
shows that  every strict deformation $R_A\to R$  is concentrated in degree $0$ and therefore  $\D_R(A)$ is naturally 
isomorphic to the set of classical deformations of the commutative algebra $R$ over the local Artin ring $A$.
\end{example}

\begin{proposition}\label{prop.strictVSlarge}
Consider $R\in\CDGA_{\K}^{\le 0}$  concentrated in degree $0$. Then for every $A\in \Art_{\K}$
there natural map $\eta_A\colon \D_R(A)\to \Def_R(A)$ is bijective with inverse
\[ H^0(-)\colon \Def_R(A)\to \D_R(A)\,. \]
\end{proposition}

\begin{proof}
For every $A\in\Art_{\K}$ consider the map 
\[\begin{split} 
H^0&\colon \Def_R(A) \longrightarrow \D_R(A),\\  
H^0&(R_A\to R)=\left(H^0(R_A)\to H^0(R_A)\otimes_A\K=H^0(R_A\otimes_A\K)\xrightarrow{\simeq}R\right),\end{split}
\]
that is properly defined since $H^0(R_A)$ is flat over $A$ by Corollary~\ref{cor.flatnessandrelation}.
On the other side, the natural map  $\eta_A\colon \D_R(A) \longrightarrow \Def_R(A)$ 
is injective since $H^0\circ \eta_A$ is the identity. Finally, again by Corollary~\ref{cor.flatnessandrelation}, 
for every $R_A\to R$ in $\Def_R(A)$ the map $R_A\to H^0(R_A)$ is a weak equivalence and this implies that 
also $\eta_A\circ H^0$ is the identity in $\Def_R(A)$.
\end{proof}

\begin{example}\label{ex.strictnogood} Strict deformations are not homotopy invariant (in any reasonable sense) for general DG-algebras.
For instance, consider the algebra $R$ in degrees $-1,0$, where 
\[ R^0=\frac{\C[x,y]}{(x^3,y^2,x^2y)},\quad R^{-1}=\C e,\qquad d(e)=x^2\quad xe=ye=0\,,\]
and notice that $R\to H^0(R)= \dfrac{\C[x,y]}{(x^2,y^2)}$ is a trivial fibration. 
We claim that there exists a first order deformation of $H^0(R)$ that does not lift to $R^0$, and therefore 
that $\D_{R}$ is not naturally isomorphic to $\D_{H^0(R)}$. 
If  $A=\C[\varepsilon]\in \Art_{\C}$ denotes the ring of dual numbers, then  the deformation 
\[ \frac{A[x,y]}{(x^2,y^2+\varepsilon)}\to H^0(R)\]
does not lift to a deformation of $R^0$. In fact the ideal $(x^3,y^2,x^2y)$ is generated by the determinants of the $2\times 2$ minors of the matrix
\[ G=\begin{pmatrix}x^2&y&0\\ 0&x&y\end{pmatrix}\]
and by Hilbert-Schaps  Theorem \cite[Thm. 5.1]{Art} every deformation of $R^0$ is induced by a deformation of the matrix $G$; in particular every first order deformation of the ideal $(x^3,y^2,x^2y)$ is contained in
the maximal ideal $(x,y)$.

\end{example}

\medskip

\subsection{Strict deformations of cofibrant DG-algebras}
\label{section.strictcofibrant}

Throughout this subsection we shall denote by $X\in\CDGA_{\K}^{\le 0}$ a cofibrant DG-algebra; then for every 
strict deformation 
\[A\xrightarrow{f_A} X_A \xrightarrow{\psi}X,\qquad A\in \DGArt_{\K}^{\le 0},\]
the map $\psi$ is surjective  and then also a fibration. Moreover, since $\K\to X_A\otimes_A\K\cong X$ is a cofibration, according to 
Corollary~\ref{cor.flattrivialcofib} also 
$f_A$ is a cofibration.

\begin{theorem}\label{thm.cfdefVSstrictcfdef}
Let $A\in\DGArt_{\K}^{\le 0}$ and consider a cofibrant object $X\in\CDGA_{\K}^{\le 0}$. Then the map
\[ \eta_A\colon \D_X(A) \to \Def_X(A) \]
is bijective.
\end{theorem}
\begin{proof} \emph{Injectivity.} Consider two strict deformations $A\to X_A\to X$ and $A\to Y_A\to X$ that are mapped in the same 
element of $\Def_X$. 
By Proposition~\ref{prop.pushoutcfdefo} there exists $A\to Z_A\to X$ in $c\Def_X(A)$ together with a commutative diagram
\[ \xymatrix{	 & A \ar@{->}[d] \ar@/_1pc/@{->}[dl] \ar@/^1pc/@{->}[dr] & \\
X_A \ar@{->}[r]^{\iota} \ar@/_1pc/@{->}[dr] & Z_A \ar@{->}[d]_{\varphi} & Y_A \ar@{->}[l]_{\sigma} \ar@/^1pc/@{->}[dl]^{\psi} \\
 & X & 	} \]
with $\sigma,\iota$ trivial cofibrations and $\varphi,\psi$ fibrations.

In order to prove that  $A\to X_A\to X$ is isomorphic to $A\to Y_A\to X$,   notice that the diagram of solid arrows
\[ \xymatrix{	Y_A \ar@{->}[r]^{id} \ar@{->}[d]_{\sigma} & Y_A \ar@{->}[d]^{\psi} \\
Z_A \ar@{->}[r]^{\varphi} \ar@{-->}[ur]^{\pi} & X	} \]
admits a  lifting $\pi\colon Z_A\to Y_A$. Therefore, the diagram
\[ \xymatrix{	 & A \ar@{->}[dl] \ar@{->}[dr] & \\
X_A \ar@{->}[rr]^{\pi\circ\iota} \ar@{->}[dr]_{\sF} & & Y_A \ar@{->}[dl]^{\psi} \\
 & X & 	} \]
commutes, and the reduction $\overline{\pi\iota}\colon X_A\otimes_A\K \to Y_A\otimes_A\K$ is an isomorphism. To conclude observe that by Corollary~\ref{cor.naka} the map $\pi\circ\iota$ is an isomorphism and the statement follows.

\emph{Surjectivity.} By Lemma~\ref{lemma.defVScdef} it is sufficient to prove that every $c$-deformation
\[ X_A\to X_A\otimes_A\K\xrightarrow{\pi} X \]
is equivalent to a strict deformation. Consider the commutative diagram
\[ \xymatrix{	 & \K \ar@/_1pc/@{->}[dl]|-{\sC} \ar@/^1pc/@{->}[dr]|-{\sC} & \\
X_A\otimes_A\K \ar@{->}[r]|-{\sC} \ar@/_1pc/@{->}[dr] & (X_A\otimes_A\K)\otimes_{\K}X \ar@{-->}[d]|-{\varphi} & X \ar@{->}[l]|-{\sC} \ar@/^1pc/@{->}[dl]|-{id_X} \\
 & X & 	} \]
in $\CDGA_A^{\le 0}$, and take a factorization of the natural map $\varphi\colon (X_A\otimes_A\K)\otimes_{\K}X\to X$ as a cofibration followed by a trivial fibration:
\[ (X_A\otimes_A\K)\otimes_{\K}X \xrightarrow{\sC} Z \xrightarrow{p} X. \]
By the \emph{2 out of 3} axiom we obtain the following commutative diagram of solid arrows
\[ \xymatrix{	X_A \ar@{->}[d] \ar@{-->}[rr]|-{\sC\sW} & & Z_A \ar@{-->}[d] & & \\
X_A\otimes_A \K \ar@{->}[rr]|-{\sC\sW} \ar@/_1pc/@{->}[drr]|-{\pi} & & Z \ar@{->}[d]^{p}_{\sF\sW} & & X \ar@{->}[ll]_{\iota}^{\sC\sW} \ar@/^1pc/@{->}[dll]|-{id_X} \\
 & & X & & 	} \]
in $\CDGA_A^{\le 0}$, where by Corollary~\ref{corollary.CW} there exists a trivial cofibration $X_A\to Z_A$ lifting $X_A\otimes_A\K \to Z$. Now observe that $e = \iota p\colon Z\to Z$ is a trivial idempotent, whose \emph{fixed locus} coincides with $X$ by Proposition~\ref{prop.fixedloci}. Moreover, by Theorem~\ref{thm.liftingidempotents} there exists a trivial idempotent $\tilde{e}\colon Z_A\to Z_A$ lifting $e$. Now consider the \emph{fixed locus} $X_A' = \lim\left\{ \xymatrix{	Z_A \ar@/^/[r]^{id} \ar@/_/[r]_{\tilde{e}} & Z_A	}\right\}$ of $\tilde{e}$ together with the natural morphism $X_A'\xrightarrow{\tilde{\iota}} Z_A$, and observe that its reduction $X_A'\otimes_A\K \to Z_A\otimes_A\K$ is $\iota\colon X\to Z$ again by Proposition~\ref{prop.fixedloci}. To conclude, consider the following commutative diagram
\[ \xymatrix{	Z_A \ar@{->}[d] & & X_A' \ar@{->}|-{\sW}[ll] \ar@{->}[d] \\
Z \ar@{->}[d] & & X \ar@{->}[ll]_{\iota}^{\sC\sW} \ar@/^1pc/@{->}[dll]|-{\id_X} \\
X & & 	} \]
which proves that $X_A'\to X\xrightarrow{id} X$ is a $c$-deformation equivalent to $Z_A\to Z\to X$, and therefore to $X_A\to X_A\otimes_A\K\to X$.
\end{proof}

\begin{remark}
Strict deformations can be generalized to an abstract strong left-proper model category $\bM$, simply replacing \emph{weak equivalence} with \emph{isomorphism} in Definition~\ref{def.deformationXL}. Notice that in the model structure of $\CDGA_{\K}^{\le 0}$ every object is fibrant. Moreover, Theorem~\ref{thm.cfdefVSstrictcfdef} essentially follows from Theorem~\ref{thm.liftingidempotents} and Corollary~\ref{corollary.CW}, which in turn can be rephrased in an abstract model category. Therefore the statement of Theorem~\ref{thm.cfdefVSstrictcfdef} can be proved in a strong left-proper model category satisfying certain additional axioms.
\end{remark}

\medskip

\subsection{Perturbation stability of cofibrations}
Throughout all this subsection we shall denote by $f\colon A\to B$ a fixed cofibration in the model category $\CDGA_{\K}^{\le 0}$.

Recall that $\CGA^{\le 0}_{\K}$ is the category of graded-commutative $\K$-algebras concentrated in non-positive degrees, and consider the natural forgetful functor 
\[\#\colon  \DGCA_{\K}^{\le 0}\to 
\CGA_{\K}^{\le 0},\qquad R\mapsto R_{\#} \;.\]
In order to avoid possible ambiguities, in the next computations it is often 
convenient to denote a DG-algebra $R$ as a pair $(R_{\#},d_R)$, where $d_R$ is the differential: in particular the morphism $f\colon A\to B$ may be also denoted by $f\colon A\to (B_{\#},d_B)$.

\begin{proposition}\label{prop.stabilitycofibration} 
Let $f\colon A\to B$ be a cofibration in $\CDGA_{\K}^{\le 0}$. Moreover, let $I\subset A$ be a differential graded nilpotent ideal and consider a derivation $\eta\in \Der^1_{A}(B,f(I)B)$ such that $(d_B+\eta)^2=[d_B,\eta]+\frac{1}{2}[\eta,\eta]=0$. Then also the morphism $f\colon A\to (B_{\#},d_B+\eta)$ is a cofibration.
\end{proposition}
\begin{proof}
The result is clear if $f$ is a semifree extension, since the semifree condition is independent of the differential. In general we can write $f$ as a weak retract of a semifree extension, i.e.
\[B\xrightarrow{i}S\xrightarrow{p}B,\qquad pi=\id_B,\quad if\colon A\to S\text{ semifree},\quad i,p\in \sW\,.\]
For simplicity of exposition, for every ideal $J\subset A$ denote $JB=f(J)B$, $JS=if(J)S$ , and 
\[ A_n=\dfrac{A}{I^n},\qquad B_n=B\otimes_A A_n=\dfrac{B}{I^nB},\qquad 
S_n=S\otimes_A A_n=\dfrac{S}{I^nS}\,.\]
Since $i,p$ are weak equivalences between cofibrant objects,  
applying the functor $-\otimes_AA_n$, for every $n$ we have  a weak retraction of cofibrant $A_n$-algebras 
\[B_n\xrightarrow{i_n}S_n\xrightarrow{p_n}B_n\,,\]
together with morphisms of short exact sequences
\begin{equation}\label{equ.coherentMC} 
\xymatrix{0\ar[r]&I^kS_n\ar[d]^{p_n}\ar[r]&S_n\ar[d]^{p_n}\ar[r]^{\alpha}&S_{k}\ar[r]\ar[d]^{p_k}&0\\
0\ar[r]&I^kB_n\ar[r]&B_n\ar[r]&B_{k}\ar[r]&0}\qquad 1\le k\le n
\end{equation}
and by the \emph{five lemma} every vertical arrow is a surjective quasi-isomorphism. For $n\ge 2$ this gives the morphism of short exact sequences   
\begin{equation}\label{equ.coherentMC2} 
\xymatrix{0\ar[r]&I^{n-1}S_n\ar[d]\ar[r]&IS_n\ar[d]^{p_n}\ar[r]^{\alpha}&IS_{n-1}
\ar[r]\ar[d]^{p_{n-1}}&0\\
0\ar[r]&I^{n-1}B_n\ar[r]&IB_n\ar[r]&IB_{n-1}\ar[r]&0}
\end{equation}
with the vertical arrows surjective quasi-isomorphisms.

Denote by $K_n$ the kernel of $p_n\colon I^{n-1}S_n\to I^{n-1}B_n$.
Notice that $\eta$ induce a coherent sequence  
$\eta_n\in \Der^*_{A_n}(B_n,IB_n)$ of solutions of the Maurer-Cartan equation in $\Der^*_{\K}(B_n,B_n)$. 

We now prove by induction on $n$ that there exists a 
coherent sequence  
$\mu_n\in \Der^*_{A_n}(S_n,IS_n)$ of solution of  the Maurer-Cartan equation in
$\Der^*_{\K}(S_n,S_n)$ such that 
\[ i_n\eta_n=\mu_n i_n,\qquad p_n\mu_n=\eta_n p_n\,.\]
This will imply that every $((B_{n})_{\#},d_{B_n}+\eta_n)$ is a retract of a semifree extension of $A_n$.

The case $n=1$ is clear since $IB_1=IS_1=0$. Now assume that $n\ge 2$ and that $\mu_{n-1}\in \Der^1_{A_{n-1}}(S_{n-1},IS_{n-1})$ as above is constructed. 

The first step is to lift $\mu_{n-1}$ to a derivation $\tau\in \Der^1_{A_{n}}(S_{n},IS_{n})$ such that $p_n\tau=\eta_n p_n$ and $\alpha \tau=\mu_{n-1}\alpha$.
The diagram \eqref{equ.coherentMC2} gives a surjective quasi-isomorphism
\[ IS_n\to IS_{n-1}\times_{IB_{n-1}}IB_n\]
and then, since $S_n$ is cofibrant the derivation  
\[(\mu_{n-1}\alpha,\eta_n p_n)\colon S_n\to IS_{n-1}\times_{IB_{n-1}}IB_n\]
can be lifted to a derivation $\tau\in \Der^1_{A_n}(S_n,IS_n)$ by 
Remark~\ref{rem.propertyofdef}.

We have $p_n(i_n\eta_n-\tau i_n)=0$ and $\alpha (i_n\eta_n-\tau i_n)=0$ and then 
$\sigma:=i_n\eta_n-\tau i_n\in \Der^1_{A_n}(B_n,K_n)$. 
Since $i_n$ is a weak equivalence of cofibrant objects, by Remark~\ref{rem.propertyofdef} 
the derivation $\sigma$ extends to $\Der^1_{A_n}(S_n,K_n)$: adding an extension of $\sigma$ to $\tau$ we can therefore assume 
\[ p_n\tau=\eta_n p_n,\qquad  \tau i_n=i_n\eta_n,\quad \alpha \tau=\mu_{n-1}\alpha\,.\] 
Finally we define
\[ r=(d_{S_n}+\tau)^2\in \Der^2_{A_n}(S_n,K_n) \;; \]
notice that $r$ is a cocycle in $\Der^2_{A_n}(S_n,I^{n-1}S_n)$ since $r(d_{S_n}+\tau)=(d_{S_n}+\tau)r$ and $r\tau=\tau r=0$ by the vanishing of $I^nS_n=0$.
Since $S_n$ is cofibrant and $K_n$ is acyclic the cocycle $r$ is a coboundary, say $r=d\psi$, and then
$\mu_{n}=\tau-\psi$ is the required solution of the Maurer-Cartan equation.
\end{proof}

\begin{remark}\label{rem.stabilitycofibration}
We shall use Proposition~\ref{prop.stabilitycofibration} in the  situation where we have a cofibration
$f\colon A\to B$, a morphism $g\colon A\to C$ in $\CDGA_{\K}^{\le 0}$, a nilpotent differential ideal 
$I\subset A$ and an \emph{isomorphism of graded algebras} $\theta\colon B\to C$ such that 
$\theta f=g$ and $\theta\colon B\otimes_A\frac{A}{I}\to C\otimes_A\frac{A}{I}$ is a morphism in 
$\CDGA_{\K}^{\le 0}$.
Then $g\colon A\to C$ is isomorphic to $f\colon A\to (B_{\#}, \theta^{-1} d_C\theta=d_B+\eta)$
for some $\eta\in \Der^1_{\K}(B, f(I)B)$. Finally, since
\[ \eta(f(a))=(d_B+\eta)(f(a))-d_B(f(a))=(d_B+\eta)(f(a))-f(d_A(a))=0 \]
for every $a\in A$ we have $\eta\in \Der^1_{A}(B, f(I)B)$ and by Proposition~\ref{prop.stabilitycofibration} also $g$ is a cofibration.
\end{remark}

\begin{remark}
Both Examples~\ref{ex.esempiononcofibrante} and \ref{ex.bruttobrutto} show that Proposition~\ref{prop.stabilitycofibration} is false in the model category $\CDGA_{\K}$ of unbounded commutative DG-algebras. In fact the morphism
\[ \K[x]\to \K[y,x],\qquad \bar{x}=1,\; \bar{y}=-1,\; dy=yx,\]
is not a $\sW$-cofibration although it can be seen as a small perturbation of the cofibration 
\[ \K[x]\to \K[y,x],\qquad \bar{x}=1,\; \bar{y}=-1,\; dy=0\,.\]
Philosophically this means that the general principle that derivations of cofibrant resolutions controls deformations is not valid in $\CDGA_{\K}$.  This was already  pointed out in \cite{H04} and  a slight modification of \cite[Example 4.3]{H04} shows that 
the functor of strict deformations $D_R\colon \Art_{\K}\to \Set$ of the unbounded cofibrant algebra 
\[R=\K[\ldots,x_{-2},x_{-1},x_0,x_1,x_2,\ldots],\qquad  \overline{x_i}=i,\; dx_i=0,\]  does not satisfy Schlessinger's conditions $(H_1), (H_2)$ of  \cite{S68}.  The result of Proposition~\ref{prop.stabilitycofibration} is assumed in \cite[4.2.2]{H04} apparently  without any additional explanations. 
\end{remark}

\bigskip
\subsection{DG-Lie algebra controlling deformations of DG-algebras}\label{section.deformationsDGalgebras}

This subsection aims to describe the differential graded Lie algebra controlling derived deformations of a commutative DG-algebra concentrated in non-positive degrees.
To this aim, the first step is the study of strict deformations of cofibrant objects; we begin by proving some preliminary results.

Let $A\in\DGArt_{\K}^{\le 0}$ and consider a morphism $A\to R_A\in\CDGA_{\K}^{\le 0}$. 
According to Corollary~\ref{cor.flattrivialcofib} we have that  $A\to R_A$ is a cofibration,
if and only if $\K\to R_A\otimes_A\K$ is a cofibration and $A\to R_A$ is flat.

\begin{proposition}\label{prop.DefCofibrantDiagram}
Let $f_A\colon A\to R_A$ be a morphism in $\CDGA_{\K}^{\le 0}$ with 
$A\in\DGArt_{\K}^{\le 0}$,  and consider the three pushout squares
\[ \xymatrix{A\ar[d]^{f_A}\ar[r]\ar@{}[dr]|(.65){\Big\ulcorner}&\K\ar[r]\ar[d]^f\ar@{}[dr]|(.65){\Big\ulcorner}&A\ar[d]^g\ar[r]\ar@{}[dr]|(.65){\Big\ulcorner}&\K\ar[d]^f\\
R_A\ar[r]^{\pi}&R\ar[r]&R\otimes_{\K}A\ar[r]^-{p}&R}\]
where the morphism $A\to \K$ in the upper row is the projection onto the residue field. 
 Then the following conditions are equivalent:
\begin{enumerate}
\item $f_A$ is a cofibration;
\item $f$ is a cofibration and there exists an isomorphism $\tilde{h}\colon (R\otimes_{\K}A)_{\#}\to (R_A)_{\#}$ of graded algebras 
such that $\pi\tilde{h}=p$ and $\tilde{h}g=f_A$.
\end{enumerate}
\end{proposition}

\begin{proof}
$(1)\Rightarrow (2)$. First notice that $f$ is a cofibration, since cofibrations are stable under pushouts. Since $\pi$ is surjective, by Lemma~\ref{lemma.liftingidempotents} the commutative diagram of solid arrows
\[ \xymatrix{	\K \ar@{->}[r] \ar@{->}[d] & R_A \ar@{->}[d]^{\pi} \\
R \ar@{->}[r]^{\id} \ar@{-->}[ur]^{h} & R	} \]
admits the dashed lifting $h\colon R\to R_A$, which is a morphism of unitary graded $\K$-algebras. By scalar extension, this gives a morphism $\tilde{h}\colon R\otimes_{\K}A\to R_A$ of graded $A$-algebras such that $\pi\tilde{h}=p$ and $\tilde{h}g=f_A$. We are only left with the proof that $\tilde{h}$ is in fact an isomorphism.

Recall that $\CDGA_{\K}^{\le 0}$ is a strong left-proper model category, so that in particular $f_A$ is flat. By induction on the length of $A$, we shall prove that the flatness of $f_A$ implies that $\tilde{h}$ is an isomorphism of graded algebras. To this aim, consider a surjective morphism $A\to B$ in $\DGArt_{\K}^{\le 0}$, and recall that choosing a cocycle $t\not= 0$ in the higher non-zero power of the maximal ideal $\mathfrak{m}_A$, we may assume the morphism $A\to B$ comes from a small extension
\[ 0 \to \K t \to A \to B \to 0 \]
in $\DGArt_{\K}^{\le 0}$; where $\K t$ is a complex concentrated in degree $i=\deg(t)$, and $\K t \to A$ is the inclusion. In fact, every surjective map in $\DGArt_{\K}^{\le 0}$ factors in a sequence of small extensions as above. Now consider the following commutative diagram of graded $A$-modules
\[ \xymatrix{	0\ar@{->}[r]  & R[-i] \ar@{->}[d]^{\id} \ar@{->}[r] & R\otimes_{\K}A \ar@{->}[d]^{\tilde{h}} \ar@{->}[r] & R\otimes_{\K}B \ar@{->}[d]^{\cong} \ar@{->}[r] & 0  \\
0\ar@{->}[r] & R[-i] \ar@{->}[r] & R_A \ar@{->}[r] & R_A\otimes_AB \ar@{->}[r] & 0	} \]
where the rows are exact, being $R_A$ an $A$-flat object. The statement follows by the five lemma.

$(2)\Rightarrow (1)$. Let $d$ and $\delta$ be the differentials of $R\otimes_{\K}A$ and $R_A$ respectively. 
The same argument used in Remark~\ref{rem.stabilitycofibration} implies that 
\[\eta=\tilde{h}^{-1}\delta\tilde{h}-d\in \Der_A^1(R\otimes_{\K}A,R\otimes_{\K}\mathfrak{m}_A)\,,\]
and then $f_A$ is isomorphic, via $\tilde{h}$, to $g\colon A\to ((R\otimes_{\K}A)_{\#}, d+\eta)$ in $\CDGA_{A}^{\le 0}$.
Since $\mathfrak{m}_A$ is a nilpotent ideal the conclusion follows by  the assumption $\pi\tilde{h}=p$ and
Proposition~\ref{prop.stabilitycofibration}.
\end{proof}

As already outlined above, we first deal with the functor of (derived) strict deformations $\D_R\colon\DGArt_{\K}^{\le 0}\to\Set$ associated to a cofibrant object $R\in\CDGA_{\K}^{\le 0}$.
To this aim, recall that to every $R\in\CDGA_{\K}^{\le 0}$ it is associated the differential graded Lie algebra $\Der^{\ast}_{\K}(R,R)$ of derivations, which in turn induces a deformation functor
\[ \Def_{\Der^{\ast}_{\K}(R,R)}\colon \DGArt_{\K}^{\le 0}\to \Set \]
as Maurer-Cartan solutions modulo gauge equivalence. In the following we shall denote by $\MC_{\Der^{\ast}_{\K}(R,R)}(A)$ the set of Maurer-Cartan elements, i.e.
\[ \MC_{\Der^{\ast}_{\K}(R,R)}(A) = \left\{ \eta\in \Der^1_{\K}(R,R)\otimes_{\K}\mathfrak{m}_A \, \vert \,  d\eta + \frac{1}{2}[\eta,\eta] = 0 \right\} \; . \]

\begin{theorem}\label{thm.DefCofibrantDiagram}
Let $R\in\CDGA_{\K}^{\le 0}$ be a cofibrant DG-algebra. Then there exists a natural isomorphism of functors
\[ \psi_1\colon \Def_{\Der^{\ast}_{\K}(R,R)} \to \D_R \]
induced by $\psi_1(\xi_A) = ((R\otimes_{\K}A)_{\#}, d_R+\xi_A)$ for every $\xi_A\in\MC_{\Der^{\ast}_{\K}(R,R)}(A)$, $A\in\DGArt_{\K}^{\le 0}$.
\begin{proof}
Fix $A\in\DGArt_{\K}^{\le 0}$ and notice that for any strict deformation $A\to R_A\to R$ in $\D_R(A)$ the map $A\to R_A$ is a cofibration. Moreover, Proposition~\ref{prop.DefCofibrantDiagram} implies that the datum of a strict deformation $A\to R_A\to R$ in $\D_R(A)$ is equivalent to a perturbation $d_R+\xi_A$ of the differential $d_R\in\Der_{\K}^1(R,R)$; which in turn corresponds to an element $\xi_A\in\Der_{\K}^1(R,R)\otimes_{\K}\mathfrak{m}_A$ such that $(d_R+\xi_A)^2 = 0$. Moreover, the integrability condition $(d_R+\xi_A)^2 = 0$ can be written in terms of the Lie structure of $\Der^{\ast}_{\K}(R,R)\otimes_{\K}\mathfrak{m}_A$:
\[ 0 = (d_R+\xi_A)^2 = d_R\xi_A + \xi_A d_R + \xi_A\xi_A = \delta(\xi_A) + \frac{1}{2}[\xi_A,\xi_A] \]
where we denoted by $\delta$ and $[-,-]$ the differential and the bracket of the DG-Lie algebra $\Der^{\ast}_{\K}(R,R)\otimes_{\K}\mathfrak{m}_A$ respectively.

The statement follows by observing that the gauge equivalence corresponds to isomorphisms of graded $A$-algebras  whose reduction to the residue field is the identity on $R$. In fact, given such an isomorphism $\varphi_A\colon R_A\to R_A'$ we can write $\varphi_A = \id + \eta_A$ for some $\eta_A\in\Hom^0_{\K}(R,R)\otimes_{\K}\mathfrak{m}_A$. Now, since $\K$ has characteristic $0$, we can take the logarithm to obtain $\varphi_A = e^{\theta_A}$ for some $\theta_A\in\Der^0_{\K}(R,R)\otimes_{\K}\mathfrak{m}_A$, see e.g. \cite[Sec. 4]{Man2}.
\end{proof}
\end{theorem}

\begin{corollary}\label{thm.deformationsDGalgebras}
Consider $X\in \CDGA_{\K}^{\le 0}$ together with a cofibrant replacement $\K \to R\xrightarrow{\pi} X$ in $\CDGA_{\K}^{\le 0}$. Then there exists a natural isomorphism of functors $ \Def_{\Der^{\ast}_{\K}(R,R)} \cong \Def_X$, which is defined on every $A\in\DGArt_{\K}^{\le 0}$ by
\[ \begin{aligned}
\psi_A\colon \Def_{\Der^{\ast}_{\K}(R,R)}(A) & \to \Def_X(A) \\
[\xi_A] & \mapsto \left[ A\to  ((R\otimes_{\K}A)_{\#}, d_R+\xi_A) \xrightarrow{\pi \circ p} X \right]
\end{aligned} \]
where $p\colon R\otimes_{\K}A\to R$ is the natural projection. 
In particular the tangent-obstruction complex of $\Def_X$ is $\Ext_{X}^*(\mathbb{L}_{X/\K},X)$, where
 $\mathbb{L}_{X/\K}\in Ho(\DGMod(X))$ denotes the cotangent complex of $X$.  
\end{corollary}

\begin{proof} The first part is an immediate consequence of Theorem~\ref{thm.DefCofibrantDiagram}, Theorem~\ref{thm.cfdefVSstrictcfdef}  and Theorem~\ref{thm.homotopyequivalence}. 
Since the cotangent complex of $X$ may be defined as 
$\mathbb{L}_{X/\K}=\Omega_{R/\K}\otimes_R X$ (\cite{Hin,QuillenCR}), the second part follows by base change formula 
\[ \Der^*_{\K}(R,R)=\Hom^*_R(\Omega_{R/\K},R)\xrightarrow{\sF\sW}
\Hom^*_R(\Omega_{R/\K},X)=\Hom^*_X(\Omega_{R/\K}\otimes_RX,X)\,.\]
\end{proof}

For readers convenience we briefly recall the geometric meaning of the  
tangent-obstruction complex for the functor $\Def_X\colon \DGArt_{\K}^{\le 0}\to \Set$, for details see \cite{Man2}: if $u$ is a variable of degree $i$ annihilated by the maximal ideal then  
$\Def_X(\K[u])=\Ext^{1-i}_X(\mathbb{L}_{X/\K},X)$, while the obstructions to lifting  deformations along a small extension 
$0\to \K u\to A\to B\to 0$ belong to $\Ext^{2-i}_X(\mathbb{L}_{X/\K},X)$.

\bigskip
\begin{acknowledgement}  This research is partially supported  by Italian MIUR under PRIN project 2015ZWST2C \lq\lq Moduli spaces and Lie theory\rq\rq. Both authors thank Ruggero Bandiera, Barbara Fantechi and Domenico Fiorenza for useful discussions on the subject of this paper.  
\end{acknowledgement}


\begin{thebibliography}{99}


\bibitem{Art} M. Artin: \emph{Deformations of Singularities.} Tata Institute of Fundamental Research, Bombay, 1976.



\bibitem{A-F} L. L. Avramov, H.-B. Foxby: \emph{Homological dimensions of unbounded complexes.} J. Pure Appl. Algebra \textbf{71}, 129-155, 1991.


\bibitem{BB} M. A. Batanin, C. Berger: \emph{Homotopy theory for algebras over polynomial monads.} Theory and Applications of Categories, 32, art. no. 6, 148-253, 2017;  	arXiv:1305.0086. 


\bibitem{BG76} 
A.K. Bousfield,  V. K. A. M. Gugenheim:  
\emph{On PL de Rham theory and rational homotopy type.} Mem. Amer. Math. Soc. \textbf{8}, 1976. 



\bibitem{kenbrown} K. S. Brown: \emph{Abstract homotopy theory and generalized sheaf cohomology.} Trans. Amer. Math. Soc. \textbf{186}, 419-458, 1973.


\bibitem{BF} R.-O. Buchweitz, H. Flenner: \emph{A Semiregularity Map for Modules and Applications to Deformations.} Compositio Mathematica 137, 135-210, 2003.









\bibitem{GS17} N. Gambino, C. Sattler: \emph{The Frobenius condition, right properness, and uniform fibrations.} Journal of Pure and Applied Algebra, 221 (12), 3027-3068, 2017.

\bibitem{GelfandManin} S.I. Gelfand, Y. I. Manin: \emph{Methods of homological algebra.} 
Springer Monographs in Mathematics, Springer-Verlag, Berlin, 2003.


\bibitem{GS} P. Goerss, K. Schemmerhorn: \emph{Model categories and simplicial methods.}  Interactions between homotopy theory and algebra, 3-49, Contemp. Math., \textbf{436}, Amer. Math. Soc., Providence, 2007. 


\bibitem{GM} W. M. Goldman, J. J. Millson. \emph{The deformation theory of representations of fundamental groups of compact K\"ahler manifolds.} Publications math\'ematiques de l'I.H.\'E.S., tome 67, no. 2, pp. 43-69, 1988.


\bibitem{derivateur} A. Grothendieck: \emph{D\'erivateurs, Ch. XIII.} \'Edition des Universit\'es de Montpellier II et Paris VII. 






\bibitem{Hin} V. Hinich: \emph{Homological algebra of homotopy algebras.} Comm. Algebra 25, 3291-3323, 1997.


\bibitem{H04} V. Hinich \emph{Deformations of Homotopy Algebras.} Communications in Algebra, Vol. 32, No. 2, pp. 473-494, 2004.


\bibitem{Hir03} P. S. Hirschhorn: \emph{Model categories and their localizations.} Mathematical Surveys and Monographs, vol. 99, American Mathematical Society, 2003.
  

\bibitem{Hov99} M. Hovey: \emph{Model categories.} Mathematical Surveys and Monographs, vol. 63, American Mathematical Society, 1999.



\bibitem{Man} M. Manetti: \emph{Differential graded Lie algebras and formal deformation theory.} Algebraic Geometry: Seattle 2005. Proc. Sympos. Pure Math. 80, 785-810, 2009.

\bibitem{Man2} M. Manetti: \emph{Extended deformation functors.} Internat. Math. Res. Notices 14, 719-756, 2002.


\bibitem{Manin17} Y. Manin: \emph{Mirrors, Functoriality, and Derived Geometry.}
arXiv:1708.02849.




\bibitem{Pal} V. P. Palamodov: \emph{Deformations of complex spaces.} Usp. Mat. Nauk 31, Nr. 3 129-194, 1976.
English translation: Russian Math. Surveys 31, 129-197, 1976.

  
\bibitem{QuillenHA} D. Quillen: \emph{Homotopical algebra}, Lecture  Notes in  Mathematics, vol. 43,  Springer-Verlag, New York, 1967.


\bibitem{QuillenCR} D. Quillen: \emph{On the (co-)homology commutative rings.} Applications of Categorical Algebra (Proc. Sympos. Pure Math., Vol. XVII, New York, 1968) Amer. Math. Soc., Providence, R.I., pp. 65-87, 1970.




\bibitem{S68} M. Schlessinger: \emph{Functors of Artin rings.} Trans. Amer. Math. Soc. 130, 208-222, 1968.


\bibitem{schwede} S. Schwede: \emph{Spectra in model categories and applications to the algebraic cotangent complex.} Journal of Pure and Applied Algebra \textbf{120} (1997) 77-104.
 




\bibitem{Sernesi} E. Sernesi:
\emph{Deformations of Algebraic Schemes.} Grundlehren der mathematischen Wissenschaften, \textbf{334}, Springer-Verlag, New York Berlin, 2006.



\bibitem{TVII} B. To\"en, G. Vezzosi: \emph{Homotopical Algebraic Geometry II: Geometric Stacks and Applications.} Memoirs of the American Mathematical Society, vol. 193, no. 902, American Mathematical Society, 2008.



\end{thebibliography}
\end{document}